\documentclass[12pt]{article}
\usepackage{hslTR}
\usepackage{fullpage}
\newenvironment{proof}
{\par\noindent\textbf{Proof}}{\eop\smallskip\vskip 3 pt}
\newcommand{\eop}{
	\hspace*{\fill}{$\vcenter{\hrule height1pt
	\hbox{\vrule width1pt height5pt
	\kern5pt \vrule width1pt} \hrule height1pt}$} }
\usepackage{graphics} 
\usepackage{graphicx} 
\usepackage{psfrag} 
\usepackage{epsfig} 
\usepackage{caption} 
\usepackage{subcaption} 
\usepackage{float} 
\usepackage{color} 
\usepackage[dvipsnames]{xcolor} 
\usepackage{paralist}

\usepackage{amsmath} 
\usepackage{latexsym} 
\usepackage{amssymb} 
\usepackage{amsxtra} 
\usepackage{abraces} 
\usepackage{mathtools}
\mathtoolsset{showonlyrefs=true}
\usepackage{subeqnarray}
\usepackage{wasysym}
\usepackage{multirow}
\usepackage{multicol}

\usepackage{enumitem} 
\usepackage[noadjust]{cite} 
\usepackage[breaklinks = true]{hyperref} 
\usepackage{ifthen}
\newcommand{\ComSp}{-1mm}
\newcommand{\ItemSp}{5mm}
\newcommand{\ItemLong}{10mm}
\newcommand{\figsize}{0.6\textwidth} 
\newcommand{\ifeq}{\text{\normalfont if }}

\newcommand{\ConfSp}[1]{\IfConf{\vspace{#1}}{ }}
\newcommand{\inter}{\text{\normalfont int}}

\newcommand{\chai}[1]{{\color{black}{#1}}}

\newcommand{\data}{(C, F, D, G)}

\newcommand{\Wc}{{\cal W}_c}
\newcommand{\Wd}{{\cal W}_d}
\newcommand{\w}{_w}
\newcommand{\Hw}{\HS\w}
\newcommand{\Cw}{C\w}
\newcommand{\Dw}{D\w}
\newcommand{\Fw}{F\w}
\newcommand{\Gw}{G\w}
\newcommand{\Lw}{L\w}
\newcommand{\hCw}{\widehat{C}\w}
\newcommand{\dataw}{(\Cw, \Fw, \Dw, \Gw)}

\newcommand{\uw}{_{u,w}} 
\newcommand{\Huw}{\HS\uw}

\newcommand{\hC}{\widehat{C}}

\newcommand{\wHS}{\widetilde{\HS}}
\newcommand{\wC}{\widetilde{C}}
\newcommand{\wF}{\widetilde{F}}
\newcommand{\wD}{\widetilde{D}}

\newcommand{\Ly}{V}
\newcommand{\level}{r}
\newcommand{\rU}{\level^*}

\newcommand{\Ls}{L_\Ly(\level)}
\newcommand{\Lss}{L_\Ly(\rU)}
\newcommand{\Ir}{{\cal I} (\level, \rU)}

\newcommand{\K}{{\cal M}_\level}

\newcommand{\Pwc}[1]{\Pi^w_c(#1)}
\newcommand{\Pwd}[1]{\Pi^w_d(#1)}
\newcommand{\wxc}{\Psi^w_c}
\newcommand{\wxd}{\Psi^w_d}

\newcommand{\z}{z}
\newcommand{\vin}{V_{in}}
\newcommand{\vdc}{V_{DC}}
\newcommand{\Ca}{C_a}
\newcommand{\vr}{v_r}

\usepackage{rgsEnvironments}
\usepackage{rgsMacros}

\graphicspath{{./Figures/}}
\newboolean{Conference}
\setboolean{Conference}{false}
\newcommand{\NotConf}[1]{\ifthenelse{\boolean{Conference}}{}{#1}}
\newcommand{\IfConf}[2]{\ifthenelse{\boolean{Conference}}{#1}{#2}}

\begin{document}
\ititle{Forward Invariance of Sets for Hybrid Dynamical Systems}
\iauthor{
  Jun Chai\\
  {\normalsize jchai3@ucsc.edu}\\
  Ricardo Sanfelice\\
    {\normalsize ricardo@ucsc.edu}}
\idate{\today{}} 
\iyear{2017}
\irefnr{012}
\makeititle
\tableofcontents
\newpage

\title{\Large \bf Forward Invariance of Sets for \\ Hybrid Dynamical Systems (Part I)}
\author{Jun Chai and Ricardo G. Sanfelice}
\maketitle
	
\begin{abstract}
In this paper, tools to study forward invariance properties with robustness to disturbances, referred to as {\em robust forward invariance}, are proposed for hybrid dynamical systems modeled as hybrid inclusions.
Hybrid inclusions are given in terms of differential and difference inclusions with state and disturbance constraints, for whose definition only four objects are required.
The proposed robust forward invariance notions allow for the diverse type of solutions to such systems (with and without disturbances), including solutions that have persistent flows and jumps, that are Zeno, and that stop to exist after finite amount of (hybrid) time.
Sufficient conditions for sets to enjoy such properties are presented.
These conditions are given in terms of the objects defining the hybrid inclusions and the set to be rendered robust forward invariant.
In addition, as special cases, these conditions are exploited to state results on nominal forward invariance for hybrid systems without disturbances.
Furthermore, results that provide conditions to render the sublevel sets of Lyapunov-like functions forward invariant are established.
Analysis of a controlled inverter system is presented as an application of our results.
Academic examples are given throughout the paper to illustrate the main ideas.
\end{abstract}
	
\section{Introduction}
\label{sec:intro}
\subsection{Background and Motivation}
Forward invariance of sets for a dynamical system are key in numerous applications, including
air traffic management \cite{tomlin1998conflict},
obstacle avoidance in vehicular networks \cite{Qi.ea.15.CDC},
threat assessment in semi-autonomous cars \cite{Falcone.ea.11.TITS},
network control systems \cite{Pin.Parisini.09.NMPC}, and
building control \cite{Meyer.ea.16.Automatica}.
Techniques to verify such properties are vital in the design of many autonomous systems.
Such tools are even more valuable under the presence of disturbances.
Formally, a set $K$ is said to be forward invariant for a dynamical system if every solution to the system from $K$ stays in $K$.
This property is also referred to as flow-invariance \cite{fernandes1987remarks} or positively invariance \cite{blanchini1999set}.
When solutions are nonunique and invariance only holds for some solutions from each point in $K$, then $K$ is said to be weakly forward invariant -- the so-called {\em viability} in \cite{aubin.viability}.
In the presence of disturbances, one is typically interested in invariance properties that hold for all possible allowed disturbances, which has been referred to as {\em robust forward invariance}; see, e.g., \cite{Meyer.ea.16.Automatica}, \chai{\cite{li2016robust,sadraddini2016provably}}.

Tools to verify invariance of a set for continuous-time and discrete-time systems have been thoroughly investigated in the literature.
In the seminal article \cite{nagumo1942lage}, the so-called {\em Nagumo Theorem} is established to determine forward invariance (and weak forward invariance) of sets for continuous-time systems with unique solutions.
Given a locally compact set $K$ that is to be rendered forward invariant and a continuous-time system with a continuous vector field, the Nagumo Theorem requires that, at each point in the boundary of $K$, the vector field belongs to the tangent cone to $K$; see also \cite[Theorem 1.2.1]{aubin.viability}.
This result has been revisited and extended in several directions.
In \cite{yorke1967invariance}, conditions for weak invariance as well as invariance for closed sets are provided -- a result guaranteeing finite-time weak invariance is also presented.
In particular, one result shows that a closed set $K$ is forward invariant for a continuous-time system with unique solutions if and only if the vector field and its negative version are subtangential to $K$ at each point in it.
A similar result is known as the Bony-Brezis theorem, which, instead of involving a condition on the tangent vectors, requires the vector field to have a nonpositive inner product with any (exterior) normal vector to the set $K$ \cite{Bony.69.AlF, Brezis.70.CPAM}.
Taking advantage of convexity and linearity of the objects considered, \cite{bitsoris1988positive} provides necessary and sufficient conditions for forward invariance of convex polyhedral sets for linear time-invariant discrete-time systems.
Essentially, conditions in \cite{bitsoris1988positive} require that the new value of the state after every iteration belongs to the set that is to be rendered forward invariant.
This condition can be interpreted as the discrete-time counterpart of the condition in the Nagumo Theorem.
For the case of time-varying continuous-time systems, \cite{fernandes1987remarks} provides conditions guaranteeing forward invariance properties of $K$ given by a sublevel set of a Lyapunov-like function; see also \cite{Zanolin.87, chellaboina1999generalized, gorban2013lyapunov}.
The analysis of forward invariance of a set for systems under the effect of perturbations has also been studied in the literature; see \cite{Tarbouriech.Burgat.94.TAC} for the case when $K$ is a cone,  \cite{Castelan.Hennet.91.CDC, Milani.Dorea.96.Automatica} when $K$ is a polyhedral.
The survey article \cite{blanchini1999set} and the book \cite{aubin.viability} summarize these and other analysis results for forward invariance of sets in continuous-time and discrete-time systems.
Thought outside the scope of this paper, the notion of robust controlled forward invariance has also been studied in the literature, see, e.g., \cite{Meyer.ea.16.Automatica, li2016robust, sadraddini2016provably}.

The study of forward invariance in systems that combine continuous and discrete dynamics is not as mature as the continuous-time and discrete-time settings.
Certainly, when the continuous dynamics are discretized, the methods for purely discrete-time systems mentioned above are applicable or can be extended without significant effort for certain classes of hybrid models in discrete time; see, in particular, the results for a class of piecewise affine discrete-time systems in \cite{rakovic2004computation}.
Establishing forward invariance (both nominal and robust) is much more involved when the continuous dynamics are not discretized.
Forward invariance of sets for impulsive differential inclusions, which are a class of hybrid systems without disturbances, are established in \cite{aubin.hs.viability}.
In particular, \cite{aubin.hs.viability} proposes conditions to guarantee (weak -- or viability -- and strong) forward invariance of closed sets and a numerical algorithm to generate invariant kernels.
For hybrid systems modeled as hybrid automata, safety specification is often recast as a forward invariance problem.
In such context, several computational approaches have been proposed for hybrid automata with nonlinear continuous dynamics, disturbances, and control inputs.
In \cite{tomlin2003computational}, a differential game approach is proposed to compute reachable sets for the verification of safety in a class of hybrid automata.
In \cite{platzer2008computing}, an algorithm is proposed to approximate invariant sets of hybrid systems that have continuous dynamics with polynomial right-hand-side and that can be written as hybrid programs.

\vspace{-0.1in}

\subsection{Contributions}
\vspace{-0.02in}
Motivated by the lack of results for the study of robust and nominal forward invariance in hybrid systems, we propose tools for analyzing forward invariance properties of sets.
In particular, formal notions of invariance and solution-independent conditions that guarantee desired invariance properties of sets are established for hybrid dynamical systems modeled as
	\begin{align}\label{eq:Hw}
		\Hw \begin{cases}
			(x, w_c) \in \Cw & \dot{x} \ \ \in \Fw(x, w_c)\\
			(x, w_d) \in \Dw & x^+ \in \Gw(x, w_d)
		\end{cases}
	\end{align}
which we refer to as hybrid inclusions \cite{65} and where $x$ is the state and $w = (w_c, w_d)$ is the disturbance; see Section~\ref{sec:Hw} for a precise definition.
In the upcoming second part of this work, tools for the design of invariance-inducing controllers for hybrid system with disturbances are proposed based on the results presented in this paper.
The main challenges in asserting such forward invariance properties of a set $K$, subset of the state space, include the following:
\begin{enumerate}[leftmargin = \ItemSp]
\item {\it Combined continuous and discrete dynamics}:
given a disturbance signal and an initial state value, a solution to \eqref{eq:Hw} may evolve continuously for some time, while at certain instances, jump. As a consequence, the set $K$ must have the property that solutions stay in it when either the continuous or the discrete dynamics are active.
\item {\it Potential nonuniqueness and noncompleteness of solutions}: 
the fact that the dynamics of \eqref{eq:Hw} are set valued and the existence of states from where flows and jumps are both allowed (namely, the state components of $\Cw$ and $\Dw$ may have a nonempty overlap with points from where flows are possible) lead to nonunique solutions.
In particular, at points in $K$ where both flows and jumps are allowed, conditions for invariance during flows and at jumps need to be enforced simultaneously.
\item {\it Presence of disturbances for systems with state constraints}:
for it to be interesting, forward invariance of a set $K$ for a hybrid system with disturbances 
is an invariance property that has to hold for all possible disturbances.
In technical terms, for every $x$ such that $(x, w_c)$ belongs to $\Cw$, the vectors in the set $\Fw(x, w_c)$ need to be in directions that flow outside of $K$ is impossible for all values of $w_c$.
Similarly, for each $x$ such that $(x, w_d)$ belongs to $\Dw$, the set $\Gw(x, w_d)$ should be contained in $K$ regardless of the values of $w_d$.
\item {\it Forward invariance analysis of intersection of sets:} when provided a Lyapunov-like function, $\Ly$, for the given system, conditions to guarantee forward invariance properties will need to take advantage of the nonincreasing property of $\Ly$.
In such a case, the state component of the sets $\Cw$ and $\Dw$ will be intersected by sublevel sets of the given Lyapunov-like function, which require less restrictive conditions than for general sets. 
\end{enumerate}

In this paper, we provide results that help tackle these key issues systematically.
For starters, we present a result to guarantee existence of nontrivial solutions to the system modeled as in \eqref{eq:Hw}, which also provides insights for solution behaviors based on completeness.
Then, we introduce the notions and sufficient conditions for forward invariance in hybrid dynamical systems.
The proposed notions of robust forward invariance are uniform over all possible disturbances, and allow for solutions to be nonunique and to cease to exist in finite (hybrid) time (namely, not complete).
For each notion, we propose sufficient conditions that the data of the hybrid inclusions and the set $K$ should satisfy to render $K$ robustly forward invariant.
Results for hybrid systems without disturbances are derived as special cases of the robust ones.
Compared to \cite{aubin.hs.viability}, which studies the nominal systems exclusively, we focus on a more general family of hybrid systems, for example, we do not always insist on the flow map to be Marchaud and Lipschitz; see, e.g., \cite{aubin.viability, aubin.hs.viability, aubin2009set}.
As an application of the results for generic sets $K$, we present a novel approach to verify forward invariance of a sublevel set of a given Lyapunov-like function intersected with the sets where continuous or discrete dynamics are allowed.
Such a result lays the groundwork for the second part of this paper.
Because of the nonincreasing properties of the given Lyapunov-like function along solutions, the developed conditions are less restrictive and more constructive when compared to the ones for a generic set $K$.
Moreover, our results are also insightful for systems with purely continuous-time or discrete-time dynamics.
In fact, because of the generality of the hybrid inclusions framework, the results in this paper are applicable to broader classes of systems, such as those studied in \cite{fernandes1987remarks,blanchini1999set,aubin.viability,yorke1967invariance, bitsoris1988positive, hu2003composite}.


\vspace{-0.2in}
\subsection{Organization and Notation}
The remainder of the paper is organized as follows.\footnote{Preliminary version of the results in this paper appeared without proof in the conference articles \cite{149} and \cite{120}. This work considers a more general class of disturbance signals than \cite{149}. Some conditions from \cite{149} to guarantee invariance are further relaxed in this paper. In addition, compared to \cite{120}, this work includes results to verify forward invariance of sublevel sets of Lyapunov-like functions.
}
Robust and nominal forward invariance are formally defined in Section~\ref{sec:HwNotions}.
Sufficient conditions to guarantee nominal and robust notions are presented in Section~\ref{sec:HSuffConds} and Section~\ref{sec:HwSuffConds}, respectively.
In Section~\ref{sec:HwLya}, a Lyapunov-like function is used to ensure forward invariance of sublevel sets.
An application of our results for the analysis of a controlled inverter system is presented in Section~\ref{sec:inverter}.
Academic examples are provided to illustrate major results.

\noindent{\bf Notation:}
Given a set-valued mapping $M : \reals^m \rightrightarrows \reals^n$, we denote the range of $M$ as $\rge M = \{ y \in \reals^n: y \in M(x), x \in \reals^m \}$ and the domain of $M$ as $\dom M = \{ x \in \reals^m : M(x) \neq \emptyset \}$. 
The closed unit ball around the origin in $\reals^n$ is denoted as $\ball$.
Given $\level \in \reals$, the $\level$-sublevel set of a function $\Ly: \reals^n \rightarrow \reals$ is $\Ls: = \{x \in \reals^n: \Ly(x) \leq \level \}$, and $\Ly^{-1}(\level) = \{x \in \reals^n: \Ly(x) = \level \}$ denotes the $\level$-level set of $\Ly$.
Given a vector $x, |x|$ denotes the 2-norm of $x$.
We use $|x|_K$ to denote the distance from point $x$ to a closed set $K$, i.e., $|x|_K = \inf\limits_{\xi \in K} |x - \xi|.$
The closure of the set $K$ is denoted as $\overline{K}$.
The set of boundary points of a set $K$ is denoted by $\partial K$ and the set of interior points of $K$ is denoted by $\inter K$.
Given vectors $x$ and $y$, $(x, y)$ is equivalent to $[x^\top \ y^\top]^\top$.

\section{Preliminaries}
\label{sec:Hw}
In this paper, we present results on robust and nominal forward invariance properties for hybrid system modeled using the hybrid inclusions framework.
More precisely, for hybrid system $\Hw$ given as in \eqref{eq:Hw}, we are interested in robust forward invariance properties of a set that are uniform in the allowed disturbances $w$;
while the notions of nominal forward invariance are studied for hybrid systems $\HS$ as in \cite{65}, which is considered as a special case of $\Hw$ with constant zero disturbance, i.e., $w \equiv 0$.
We further explore the relaxed conditions to guarantee nominal forward invariance of sublevel sets of Lyapunov-like functions.
In this section, we present basic definitions and properties of $\Hw$ that are important for deriving the forthcoming results.

Following \cite{65}, a solution to the hybrid system $\Hw$ is parameterized by the ordinary time variable $t \in \reals_{\geq 0}:=[0,\infty)$ and by the discrete jump variable $j \in \nats := \{0,1,2,...\}$, and defined on a hybrid time domain $\SSS \subset \reals_{\geq 0} \times \nats$; see \cite[Definition 2.3]{65}.
The set $\SSS$ is a hybrid time domain if, for each $(T,J)\in \SSS$, $\SSS\ \cap\ \left( [0,T]\times\{0,1,...,J\} \right)$ can be written as $\cup_{j=0}^{J-1} \left([t_j,t_{j+1}],j\right)$ for some finite sequence of times $0=t_0\leq t_1 \leq t_2 \leq ... \leq t_J$.
A hybrid arc $\phi$ is a function on a hybrid time domain if, for each $j\in\nats$, $t\mapsto \phi(t,j)$ is absolutely continuous on the interval $I^j: = \{t : (t,j) \in \dom \phi \}$. 

The data of hybrid system $\Hw$ in \eqref{eq:Hw} is defined by the flow set $\Cw\subset \reals^n \times \Wc$, the flow map $\Fw : \reals^n \times \Wc \rightrightarrows \reals^n$, the jump set $\Dw\subset \reals^n \times \Wd$, and the jump map $\Gw : \reals^n \times \Wd \rightrightarrows \reals^n$.
The space for the state $x$ is $\reals^n$ and the space for the disturbance $w = (w_c, w_d)$ is ${\cal W} = \Wc \times \Wd \subset \reals^{d_c}\times \reals^{d_d}$.
The sets $\Cw$ and $\Dw$ define conditions that $x$ and $w$ should satisfy for flows or jumps to occur.
In this paper, we assume $\dom \Fw \supset \Cw$ and $\dom \Gw \supset \Dw$.
A hybrid disturbance $w$ is a function on a hybrid time domain that, for each $j \in \nats$, $t \mapsto w(t,j)$ is Lebesgue measurable and locally essentially bounded on the interval $\{t : (t, j) \in \dom w \}$.
When $w(t, j) = 0$ for every $(t,j) \in \dom w$ (which means that there is no disturbance), the system $\Hw$ reduces to the nominal hybrid system introduced in \cite{65}, which is given by
	\begin{align}\label{eq:H}
		\HS \begin{cases}
			x \in C &\dot{x} \ \ \in F(x)\\
			x \in D &x^+ \in G(x).
		\end{cases}
	\end{align}

For convenience, we define the projection of $S \subset \reals^n \times \Wc$ onto $\reals^n$ as
$\Pwc{S}:= \{x \in \reals^n: \exists w_c \in \Wc \text{ s.t. }(x, w_c) \in S\},$
and the projection of $S \subset \reals^n \times \Wd$ onto $\reals^n$ as
$\Pwd{S}:= \{x \in \reals^n: \exists w_d \in \Wd \text{ s.t. }(x, w_d) \in S\}.$
Given sets $\Cw$ and $\Dw$, the set-valued maps $\wxc: \reals^n \rightrightarrows \Wc$ and $\wxd: \reals^n \rightrightarrows \Wd$ are defined for each $x \in \reals^n$
as $\wxc(x) := \{w_c \in \Wc: (x, w_c) \in \Cw\}$
and $\wxd(x) := \{w_d \in \Wd: (x, w_d) \in \Dw\}$, respectively.

As an extension to the definition of solution to \eqref{eq:H}, namely, Definition~\ref{def:solution}, solution pairs to a hybrid system $\Hw$ as in \eqref{eq:Hw} are defined as follows.
\begin{definition}(solution pairs to $\Hw$)\label{def:HwSol}
	A pair $(\phi, w)$ consisting of a hybrid arc $\phi$ and a hybrid disturbance $w = (w_c, w_d)$, with $\dom\phi = \dom w (= \dom (\phi, w)),$ is a solution pair to the hybrid system $\Hw$ in \eqref{eq:Hw} if $(\phi(0,0), w_c(0,0)) \in \overline{\Cw}$ or $(\phi(0,0), w_d(0,0))\in \Dw$, and
	\begin{enumerate}[label = (S$\w$\arabic*), leftmargin = \ItemLong]
		\item\label{item:Sw1} for all $j \in \nats$ such that $I^j$ has nonempty interior
		\begin{align*}
		&(\phi(t,j), w_c(t,j)) \in \Cw &\mbox{for all } t \in \text{\normalfont int }I^j,\\
		&\frac{d\phi}{dt}(t,j) \in \Fw(\phi(t,j), w_c(t,j)) &\mbox{for almost all } t \in I^j,
		\end{align*}
		\item\label{item:Sw2}  for all $(t,j) \in \dom\phi$ such that $(t, j+1) \in \dom\phi$,
		\begin{align*}
		&(\phi(t,j), w_d(t,j)) \in \Dw\\
		&\phi(t,j+1) \in \Gw(\phi(t,j), w_d(t,j)).
		\end{align*}
	\end{enumerate}
	In addition, a solution pair $(\phi, w)$ to $\Hw$ is
	\begin{itemize}[leftmargin = \ItemSp]
		\item nontrivial if $\dom (\phi,w)$ contains at least two points;
		\item complete if $\dom (\phi,w)$ is unbounded;
		\item maximal if there does not exist another $(\phi, w)'$ such that $(\phi, w)$ is a truncation of $(\phi, w)'$ to some proper subset of $\dom(\phi, w)'$.
		\hfill $\square$
	\end{itemize}
\end{definition}
We use $\sol_{\Hw}$ to represent the set of all maximal solution pairs to the hybrid system $\Hw$ and, given $K \subset \reals^n$,  $\sol_{\Hw}(K)$ denotes the set that includes all maximal solution pairs $(\phi, w)$ to the hybrid system $\Hw$ with $\phi(0,0) \in K$.


\section{Robust and Nominal Forward Invariance}
\label{sec:HwNotions}
In this section, we formally introduce the notions of robust and nominal forward invariance of sets for system $\Hw$ given as in \eqref{eq:Hw} and $\HS$ given as in \eqref{eq:H}, respectively.
In particular, a set $K$ enjoys robust forward invariance when the state evolution begins from $K$ and stays within $K$ regardless of the value of the disturbance $w$.
First, we introduce weak versions of such forward invariance notions for $\Hw$.
\begin{definition}(robust weak forward (pre-)invariance of a set) \label{def:rwFI}
	The set $K \subset \reals^n$ is said to be robustly weakly forward pre-invariant for $\Hw$ if for every $x \in K $ there exists one solution pair $(\phi, w) \in \sol_{\Hw}(x)$ such that $\rge \phi \subset K$.
	The set $K \subset \reals^n$ is said to be robustly weakly forward invariant for $\Hw$ if for every $x \in K$ there exists a complete $(\phi, w)\in \sol_{\Hw}(x)$ such that $\rge \phi \subset K$.
	\hfill $\square$
\end{definition}

The following notions are considered stronger than the ones in Definition~\ref{def:rwFI} because all maximal solution pairs that start from the set $K$ are required to stay within $K$.
\begin{definition}(robust forward (pre-)invariance of a set) \label{def:rFI}
	The set $K \subset \reals^n$ is said to be robustly forward pre-invariant for $\Hw$ if every $(\phi, w)\in \sol_{\Hw}(K)$ is such that $\rge \phi \subset K$.
	The set $K \subset \reals^n$ is said to be robustly forward invariant for $\Hw$ if for every $x \in K$ there exists a solution pair to $\Hw$ and every $(\phi, w)\in \sol_{\Hw}(K)$ is complete and such that $\rge \phi \subset K$.
	\hfill $\square$
\end{definition}

In the upcoming second part of this paper, Definition~\ref{def:rwFI} and Definition~\ref{def:rFI} are presented in the context of {\em robust controlled invariance} properties of sets for $\Huw$ under the effect of a given state-feedback pair.

For hybrid systems without disturbances, i.e., $\HS$ given as in \eqref{eq:H}, the forward invariance notions introduced above reduce to the ones below; see also \cite[Definition 2.3 - Definition 2.6]{120}.
\begin{definition}(nominal forward invariance of a set)\label{def:FI}
	The set $K \subset \reals^n$ is said to be
	\begin{itemize}[leftmargin = 4mm]
		\item weakly forward pre-invariant for $\HS$ if for every $x \in K$ there exists one $\phi \in \sol_{\HS}(x)$ with $\rge\phi \subset K$;
		\item weakly forward invariant for $\HS$ if for every $x \in K$ there exists one complete solution $\phi \in \sol_{\HS}(x)$ with $\rge\phi \subset K$;
		\item forward pre-invariant for $\HS$ if every $\phi \in \sol_{\HS}(K)$ has $\rge\phi \subset K$;
		\item forward invariant for $\HS$ if for every $x \in K$ there exists one solution, and every $\phi \in \sol_{\HS}(K)$ is complete and has $\rge\phi \subset K$.
	\end{itemize}
\end{definition}
The relationship among these four notions is summarized in the diagram in Figure~\ref{fig:notions}.
\IfConf{
\vspace{-3mm}

	\begin{figure}[H]
		\hspace{2mm}
		\setlength{\unitlength}{.4\textwidth}
		\includegraphics[width=\unitlength]{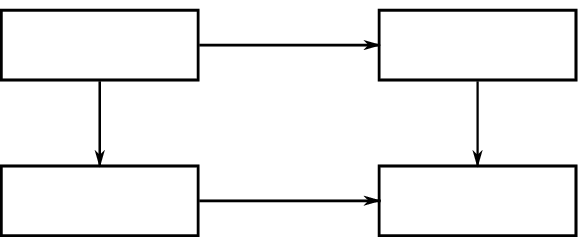}
		\small
		\put(-.16,.2){$+$ Completeness}
		\put(-.82,.2){$+$ Completeness}
		\put(-.64,.36){$+ \forall \phi \in \sol_{\HS}(K)$}
		\put(-.64,.09){$+ \forall \phi \in \sol_{\HS}(K)$}
		\put(-.96,.35){Weak Forward}
		\put(-.95,.3){pre-Invariance}
		\put(-.96,.08){Weak Forward}
		\put(-.93,.03){ Invariance}
		\put(-.25,.35){Forward}
		\put(-.3,.3){pre-Invariance}
		\put(-.25,.08){Forward}
		\put(-.28,.03){ Invariance}
	\vspace{-1mm}
		\caption{Relationships of the notions of forward invariance for a set $K$.}
		\label{fig:notions}
	\end{figure}

\vspace{-4mm}

}{
	\begin{figure}[H]
		\centering
		\setlength{\unitlength}{.6\textwidth}
		\includegraphics[width=\unitlength]{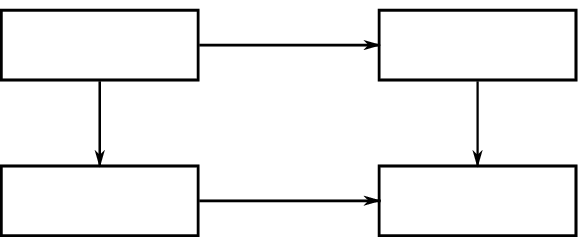}
		\small
		\put(-.16,.2){$+$ Completeness}
		\put(-.82,.2){$+$ Completeness}
		\put(-.64,.36){$+ \forall \phi \in {\cal S}_{\HS}(K)$}
		\put(-.64,.09){$+ \forall \phi \in {\cal S}_{\HS}(K)$}
		\put(-.96,.35){Weak Forward}
		\put(-.95,.3){pre-Invariance}
		\put(-.96,.08){Weak Forward}
		\put(-.93,.03){ Invariance}
		\put(-.25,.35){Forward}
		\put(-.3,.3){pre-Invariance}
		\put(-.25,.08){Forward}
		\put(-.28,.03){ Invariance}
		\caption{Relationships of the notions of forward invariance for a set $K$.}
		\label{fig:notions}
	\end{figure}
}
Note that some of the proposed notions require solutions to exist and maximal solutions to be complete.
Hence, inspired by the conditions guaranteeing existence of solutions to $\HS$ (see Proposition~\ref{prop:existence}), we provide the following result for guaranteeing existence of nontrivial solution pairs to $\Hw$ and characterizing their possible behavior.
\begin{proposition}(basic existence under disturbances)\label{prop:wexistence}
	Consider a hybrid system $\Hw = \dataw$ as in \eqref{eq:Hw}.
	Let $\xi \in \overline{\Pwc{\Cw}} \cup \Pwd{\Dw}$.
	If $\xi \in \Pwd{\Dw}$, or
	\begin{enumerate}[label = (VC$\w$), leftmargin=3\parindent]
		\item\label{item:VCw} there exist $ \varepsilon > 0$, an absolutely continuous function $\widetilde{z}: [0,\varepsilon] \rightarrow \reals^n$ with $\widetilde{z}(0) = \xi$, and a Lebesgue measurable and locally essentially bounded function $\widetilde{w}_c : [0,\varepsilon] \rightarrow \Wc$ such that $(\widetilde{z}(t), \widetilde{w}_c(t))\in \Cw$ for all $t \in (0,\varepsilon)$ and $\dot{\widetilde{z}}(t) \in \Fw(\widetilde{z}(t), \widetilde{w}_c(t))$ for almost all $t \in [0,\varepsilon]$, where $\widetilde{w}_c(t) \in \wxc(\widetilde{z}(t))$ for every $t \in [0,\varepsilon]$,
	\end{enumerate}
	then, there exists a nontrivial solution pair $(\phi, w)$ from the initial state $\phi(0,0) = \xi$.
	If $\xi \in \Pwd{\Dw}$ and \ref{item:VCw} holds for every $\xi \in \overline{\Pwc{\Cw}} \setminus \Pwd{\Dw}$, then there exists a nontrivial solution pair to $\Hw$ from every initial state $\xi \in \overline{\Pwc{\Cw}} \cup \Pwd{\Dw}$, and every solution pair $(\phi, w) \in \sol_{\Hw}(\overline{\Pwc{\Cw}} \cup \Pwd{\Dw})$ from such points satisfies exactly one of the following:
	\begin{enumerate}[label = \alph*), leftmargin = \ItemSp]
		\item\label{item:a} the solution pair $(\phi,w)$ is complete;
		\item\label{item:b} $(\phi,w)$ is not complete and ``ends with flow'': with $(T,J) = \sup\dom(\phi, w)$, the interval $I^J$ has nonempty interior, and either
		\begin{enumerate}[label = b.\arabic*)]
			\item\label{item:b.1} $I^J$ is closed, in which case either
			\begin{enumerate}[label = b.1.\arabic*)]
				\item\label{item:b.1.1} $\phi(T,J) \in \overline{\Pwc{\Cw}}\setminus(\Pwc{\Cw} \cup \Pwd{\Dw})$, or
				\item\label{item:b.1.2} from $\phi(T,J)$ flow within $\Pwc{\Cw}$ is not possible, meaning that there is no $\varepsilon > 0$, absolutely continuous function $\widetilde{z}: [0,\varepsilon] \rightarrow \reals^n$ and a Lebesgue measurable and locally essentially bounded function $\widetilde{w}_c : [0,\varepsilon] \rightarrow \Wc$ such that $\widetilde{z}(0) = \phi(T,J)$, $(\widetilde{z}(t), \widetilde{w}_c(t))\in \Cw$ for all $t \in (0,\varepsilon)$, and $\dot{\widetilde{z}}(t) \in \Fw(z(t), \widetilde{w}_c(t))$ for almost all $t \in [0,\varepsilon]$, where $\widetilde{w}_c(t) \in \wxc(\widetilde{z}(t))$ for every $t \in [0,\varepsilon]$, or
			\end{enumerate}
			\item\label{item:b.2} $I^J$ is open to the right, in which case $(T,J) \notin \dom(\phi, w)$ due to the lack of existence of an absolutely continuous function $\widetilde{z} : \overline{I^J} \rightarrow \reals^n$ and a Lebesgue measurable and locally essentially bounded function $\widetilde{w}_c : [0,\varepsilon] \rightarrow \Wc$ satisfying $(\widetilde{z}(t), \widetilde{w}_c(t)) \in \Cw$ for all $t \in \inter I^J$, $\dot{\widetilde{z}}(t) \in  \Fw(\widetilde{z}(t), \widetilde{w}_c(t))$ for almost all $t \in I^J,$ and such that $\widetilde{z}(t) = \phi(t,J)$ for all $t \in I^J $, where $\widetilde{w}_c(t) \in \wxc(\widetilde{z}(t))$ for every $t \in [0,\varepsilon]$;
		\end{enumerate}
		\item\label{item:c} $(\phi, w)$ is not complete and ``ends with jump'': with $(T,J) = \sup\dom(\phi, w) \in \dom(\phi, w)$, $(T,J-1) \in \dom (\phi, w)$, and either
		\begin{enumerate}[label = c.\arabic*)] 
			\item\label{item:c.1} $\phi(T,J) \notin \overline{\Pwc{\Cw}} \cup \Pwd{\Dw}$, or
			\item\label{item:c.2} $\phi(T,J) \in \overline{\Pwc{\Cw}}$,\footnote{As a consequence of $\phi(T,J) \notin \Pwd{\Dw}$, $\phi(T,J) \in \overline{\Pwc{\Cw}}\setminus \Pwd{\Dw}$ is under the condition in case \ref{item:c.2}.} and from $\phi(T,J)$ flow within $\Pwc{\Cw}$ as defined in \ref{item:b.1.2} is not possible.
		\end{enumerate}
	\end{enumerate}
\end{proposition}
	\begin{proof}
	To prove the existence of a nontrivial solution pair from $\xi$, we show that under the given assumptions, a solution pair $(\phi,w)$ satisfying the conditions in Definition~\ref{def:HwSol} can be constructed such that $\dom(\phi, w)$ contains at least two points.
	We have the following cases:
	\begin{enumerate}[label = \roman*), leftmargin = \ItemSp]
		\item\label{item:wexisti}  If $\xi \in \Pwd{\Dw}$, then there exist $w^*_d$ such that $(\xi, w^*_d) \in \Dw$ by definition of $\Pwd{\Dw}$.
		Let the hybrid disturbance $w_1 = (w_c, w_d)$ be defined on $\dom w_1 := \{(0,0)\} \cup \{(0,1)\}$ as $w_d(0,0) = w^*_d$ and $w_d(0,1) = a$, where $a \in \Wd$ and $w_c$ can be arbitrary.
		By definition of the jump map $\Gw$, there exists $b \in \Gw(\xi, w^*_d)$.
		Let $\phi_1$ be a hybrid arc with $\dom \phi_1 = \dom w_1$ defined as $\phi_1(0,0) = \xi$ and $\phi_1(0,1) = b$.
		Then, $(\phi_1, w_1)$ is a nontrivial solution pair to $\Hw$;
		\item\label{item:wexistii}  If $\xi \in \overline{\Pwc{\Cw}} \setminus \Pwd{\Dw}$ and \ref{item:VCw} holds, there exist $\varepsilon > 0$, an absolutely continuous function $\widetilde{z}: [0, \varepsilon] \rightarrow \reals^n$ and a Lebesgue measurable and locally essentially bounded function $\widetilde{w}_c : [0,\varepsilon] \rightarrow \Wc$ with $\widetilde{z}(0) = \xi$ and $\widetilde{w}_c(0) \in \wxc(\xi)$ satisfying \ref{item:Sw1} in Definition~\ref{def:HwSol}.
		Let the hybrid disturbance $w_2 = (w_c, w_d)$ be defined on $\dom w_2 := [0, \varepsilon) \times \{0\}$ with $w_c(t,0) = \widetilde{w}_c(t)$ for every $t \in [0,\varepsilon)$ and let $w_d$ be given arbitrarily.
		Let the hybrid arc $\phi_2$ be defined on $\dom \phi_2 = \dom w_2$ as $\phi_2(t,0) = \widetilde{z}(t)$ for every $t \in [0,\varepsilon)$.
		Then, $(\phi_2, w_2)$  is a nontrivial solution pair to $\Hw$.
	\end{enumerate}
	Item \ref{item:wexisti}  and \ref{item:wexistii}  imply the existence of a nontrivial solution pair to $\Hw$ from every $\xi \in \Pwd{\Dw}$ and every $\xi \in \overline{\Pwc{\Cw}}\setminus \Pwd{\Dw}$, respectively, that is, for every $\xi \in \overline{\Pwc{\Cw}} \cup \Pwd{\Dw}$.
	
	Next, we prove that every maximal solution pair $(\phi,w)$ to $\Hw$ satisfies exactly one of the properties in \ref{item:a}, \ref{item:b}, and \ref{item:c}.
	Suppose the nontrivial solution pair $(\phi, w)$ is not complete, i.e., case \ref{item:a} does not hold and either \ref{item:b} or \ref{item:c} holds.
	We show that only one of these properties holds.
	Let $(T,J) = \sup\dom(\phi, w)$.
	
	If $(T,J) \in \dom(\phi, w)$, then $I^J$ is closed and case \ref{item:b.2} does not hold, for which we have either
	\begin{enumerate}[label = \roman*), resume, leftmargin = \ItemSp]
		\item\label{item:wexistiii}  $I^J$ is a singleton; or
		\item\label{item:wexistiv}  $I^J$ has nonempty interior.
	\end{enumerate}
	If \ref{item:wexistiii}  is true, the solution pair $(\phi, w)$ ends with a jump and either $\phi(T,J) \notin \overline{\Pwc{\Cw}} \cup \Pwd{\Dw}$, which directly leads to case \ref{item:c.1}, or $\phi(T,J) \in \overline{\Pwc{\Cw}} \cup \Pwd{\Dw}$.
	The latter case leads to \ref{item:c.2} only since otherwise $(\phi, w)$ can be extended by flow via the functions $\widetilde{z}$ and $\widetilde{w}_c$ as described in \ref{item:b.1.2} or by a jump as described in item \ref{item:wexisti} above with an arbitrary $w_d \in \wxd(x)$.
	If \ref{item:wexistiv} is true, then, by item \ref{item:Sw1} in Definition~\ref{def:HwSol}, case \ref{item:b.1.1} holds, i.e., $\phi(T,J) \in \overline{\Pwc{\Cw}}\setminus (\Pwc{\Cw} \cup\Pwd{\Dw})$, or case \ref{item:b.1.2} holds, namely, the solution pair $(\phi, w)$ cannot be extended via flows.
	In summary, if $(T,J) \in \dom \phi$, then only one among \ref{item:b.1.1}, \ref{item:b.1.2}, \ref{item:c.1} and \ref{item:c.2} may hold.
	
	If $(T,J) \notin \dom (\phi, w)$, then $I^J$ is open to the right, and by maximality of $(\phi, w)$, \ref{item:b.2} holds.
\end{proof}

\ConfSp{-2mm}
Proposition~\ref{prop:wexistence} presents conditions guaranteeing existence of nontrivial solution pairs to $\Hw$ from every initial state $\xi \in \overline{\Pwc{\Cw}}\cup \Pwd{\Dw}$, as well as characterizes all possibilities for maximal solution pairs.
In particular, maximal solution pairs that are not complete can either ``end with flow'' or ``end with jump.''
In short, the former means that $I^J$ has a nonempty interior over which $(\phi(t,J), w_c(t,J)) \in \Cw$ for all $t\in \inter I^J$ and $\frac{d\phi}{dt}(t,J) \in \Fw(\phi(t,J), w_c(t,J))$ for almost all $t\in \inter I^J$, where $(T,J) = \sup\dom (\phi, w)$.
In particular, case \ref{item:b.1.1} corresponds to a solution pair ending at the boundary of $\overline{\Cw}$, case \ref{item:b.1.2} describes the case of a solution pair ending after flowing and at a point, where continuing to flow is not possible, while case \ref{item:b.2} covers the case of a solution pair escaping to infinity in finite time.
The case ``end with jump'' means that $(T,J), (T, J - 1) \in \dom (\phi, w), (\phi(T,J-1), w_d(T,J-1)) \in \Dw$, and the solution pair ends either with $\phi(T,J) \in \Pwc{\Cw} \cup \Pwd{\Dw}$ due to flow being not possible or with $\phi(T,J) \notin \Pwc{\Cw} \cup \Pwd{\Dw}$, where $(T,J) = \sup\dom (\phi,w)$.

\begin{remark}\label{rem:existence}
	Case \ref{item:c.1} in Proposition~\ref{prop:wexistence} is not possible when $\Gw(\Dw) \subset \overline{\Pwc{\Cw}} \cup \Pwd{\Dw}$.\footnote{$\Gw(\Dw) = \{x' \in \reals^n: \exists (x, w_d) \in \Dw, x' \in \Gw(x,w_d) \}$}
	Moreover, when the disturbance signal $w_c$ is generated by an exosystem of the form\footnote{The disturbance $w_c$ generated by \eqref{eq:wc} are not necessarily differentiable but rather, absolutely continuous over each interval of flow.
	For examples of exosystems given as in \eqref{eq:wc} and having also jumps, see \cite{46}.}
\ConfSp{-4mm}
	\begin{align}\label{eq:wc}
		\dot{w_c} \in F_e(w_c) \quad w_c \in \Wc,
	\end{align}

\ConfSp{-2mm}
\noindent \ref{item:VCw} can be guaranteed if, for each $(\xi, w'_c)$, there exists a neighborhood $U$ such that for every $(x, w_c) \in U \cap \Cw, (\Fw(x, w_c), F_e(w_c)) \cap T_{\Cw}(x, w_c) \neq \emptyset$, provided that $\Cw$ is closed and $(\Fw, F_e)$ is outer semicontinuous and locally bounded with nonempty and convex values on $\Cw$.
\end{remark}

%
%

\section{Sufficient Conditions for Robust and Nominal Forward Invariance}
\label{sec:SuffConds}
Definition~\ref{def:rwFI} and Definition~\ref{def:rFI} state that a set enjoys robust forward invariance properties when the state evolution stays within the set regardless of the value of the disturbance.
When disturbance signals are identically zero, Definition~\ref{def:FI} reduces to nominal forward invariance properties for $\HS$ given as in \eqref{eq:H}.
Verifying these properties for a given set using the definitions requires to solve for solution pairs to $\Hw$, and solutions to $\HS$, respectively, explicitly.
To circumvent that challenge, we present, when possible, solution-independent conditions to guarantee each notion.

For clarity of exposition, in Section~\ref{sec:HSuffConds}, we provide sufficient conditions for the nominal cases, of which the preliminary version is presented in \cite{120}.
Then, in Section~\ref{sec:HwSuffConds}, these conditions are extended to hybrid systems with generic disturbance signals, i.e., $\Hw$ given as in \eqref{eq:Hw}.

\ConfSp{-6mm}

\subsection{Sufficient Conditions for Nominal Forward Invariance Properties for $\HS$}
\label{sec:HSuffConds}
\ConfSp{-1mm}
We present the sufficient conditions for forward invariance of a given set $K$ for $\HS$ that involve the data $\data$.
For the discrete dynamics, namely, the jumps, such conditions involve the understanding of where $G$ maps the state to.
Inspired by the well-known Nagumo Theorem \cite{nagumo1942lage}, for the continuous dynamics, namely, the flows, our conditions use the concept of tangent cone to the closed set $K$.
The tangent cone at a point $x \in \reals^n$ of a closed set $K \subset \reals^n$ is defined using the Dini derivative of the distance to the set, and is given by\footnote{In other words, $\omega$ belongs to $T_K(x)$ if and only if there exist sequences $\tau_i \searrow 0$ and $\omega_i \rightarrow \omega$ such that $x + \tau_i \omega_i \in K$ for all $i \in \nats$; see also \cite[Definition 1.1.3]{aubin.viability}. The latter property is further equivalent to the existence of sequences $x_i \in K$ and $\tau_i > 0$ with $x_i \rightarrow x, \tau_i \searrow 0$ such that $\omega = \lim_{i\to\infty} (x_i - x)/\tau_i$.}
\ConfSp{-4mm}
	\begin{align}\label{eq:tangent}
		T_K(x) = \left\{ \omega \in \reals^n : \liminf\limits_{\tau \searrow 0} \frac{|x+\tau \omega|_K}{\tau} = 0 \right\}.
	\end{align}
In the literature (see, e.g., \cite[Definition 4.6]{jahn1994tangent} and \cite{aubin.hs.viability}), this tangent cone is also known as the sequential Bouligand tangent cone or contingent cone.
In contrast to the Clarke tangent cone introduced in \cite[Remark 4.7]{jahn1994tangent}, which is always a closed convex cone for every $x \in K$, the tangent cone (possibly nonconvex) we consider in this work includes all vectors that point inward to the set $K$ or that are tangent to the boundary of $K$.\footnote{Note that, for a convex set, the Bouligand tangent cone coincides with the Clarke tangent cone.}

Our sufficient conditions for forward invariance require part of the data of $\HS$ and the set $K$ to satisfy the following mild assumption.
\begin{assumption}\label{assum:data}
	The sets $K, C$, and $D$ are such that $K \subset \overline{C}\cup D$ and that $K \cap C$ is closed.
	The map $F:\reals^n \rightrightarrows \reals^n$ is outer semicontinuous, locally bounded relative to $K \cap C$, and $F(x)$ is nonempty and convex for every $x \in K \cap C$.
\end{assumption}

The following result is a consequence of the forthcoming Theorem~\ref{thm:rwFI} in Section~\ref{sec:HwSuffConds} and its proof will be delayed to that section.
Sufficient conditions for a given set $K$ to be weakly forward pre-invariant and weakly forward invariant are presented.

\begin{theorem}(nominal weak forward pre-invariance and weak forward invariance)\label{thm:wfi}
	Given $\HS= \data$ as in \eqref{eq:H} and a set $K$, suppose $K, C, D,$ and $F$  satisfy Assumption~\ref{assum:data}.
	The set $K$ is weakly forward pre-invariant for $\HS$ if the following conditions hold:
	\begin{enumerate}[label = \ref{thm:wfi}.\arabic*), leftmargin = 12mm]
		\item\label{item:wfi1} For every $x \in K \cap D$, $G(x) \cap K \neq \emptyset$;
		\item\label{item:wfi2} For every $x \in \chai{\hC} \setminus D$, $F(x) \cap T_{K \cap C}(x) \neq \emptyset$;
	\end{enumerate}
where $\chai{\hC} : = \partial(K \cap C) \setminus L$ and $L : = \{x \in \partial C: F(x) \cap T_{\overline{C}}(x) = \emptyset \}$.
	Moreover, $K$ is weakly forward invariant for $\HS$ if, in addition\chai{, $K \cap L \subset D$ and, with $K^\star = K \setminus D$,}
	\begin{enumerate}[label = N$\star$), resume, leftmargin = 9mm]
		\item\label{item:N*} for every $\phi \in \sol_{\HS}( \chai{K^\star})$ with $\rge \phi \subset K$,
		case \ref{item:b.2n} in Proposition~\ref{prop:existence} does not hold.
	\end{enumerate}
\end{theorem}
\ConfSp{-1mm}

The next result, which is a consequence of Theorem~\ref{thm:rFI} provides sufficient conditions for a set $K$ to be forward pre-invariant and forward invariant for $\HS$.
\begin{theorem}(nominal forward pre-invariance and forward invariance)\label{thm:fi}
	Given $\HS= \data$ as in \eqref{eq:H} and a set $K \subset \reals^n$. Suppose $K, C, D,$ and $F$ satisfy Assumption~\ref{assum:data} and that $F$ is locally Lipschitz on $C$.
	Let $\hC$ and $L$ be given as in Theorem~\ref{thm:wfi}.
	The set $K$ is forward pre-invariant for $\HS$ if the following conditions hold:
	\begin{enumerate}[label = \ref{thm:fi}.\arabic*), leftmargin=12mm]
		\item\label{item:fi1} $G(K \cap D) \subset K$;
		\item\label{item:fi2} For every $x \in \hC$, $F(x) \subset T_{K \cap C}(x)$.
	\end{enumerate}
	Moreover, $K$ is forward invariant for $\HS$ if, in addition, $K \cap L \subset D$ and, with $K^\star = K\cap C$, item \ref{item:N*} in Theorem~\ref{thm:wfi} holds.
\end{theorem}

\chai{
\begin{remark}
	Some of the conditions in Theorem~\ref{thm:wfi} and Theorem~\ref{thm:fi} are weaker than those required by results in \cite{120}.
	The construction of the set $L$ in items \ref{item:wfi2} and \ref{item:fi2} is inspired by the {\em viability domain} in \cite[Definition 1.1.5]{aubin.viability}.
	Note that when \ref{item:N*} holds, completeness of maximal solutions is guaranteed by ensuring that
	$K \cap L \subset D$, which guarantees that solutions can continue to evolve from $L$ via a jump.
\end{remark}
}

The following example is used to illustrate Theorem~\ref{thm:wfi} and Theorem~\ref{thm:fi}.
\begin{example}[solutions with finite escape time]\label{ex:pre}
	Consider the hybrid system $\HS = \data$ in $\reals^2$ with system data given by
	\IfConf{
	$F(x) :=(1 + x_1^2, 0)$ for every $x \in C:= \{ x \in \reals^2: x_1 \in [0, \infty), x_2 \in [-1,1] \}$ and $G(x) := (x_1+ \ball, x_2)$ for every $x \in D:= \{ x \in \reals^2: x_1 \in [0, \infty), x_2 = 0 \}.$
	}{
	\begin{align*}
	& F(x) := \begin{bmatrix} 1 + x_1^2\\ 0 \end{bmatrix} 
	& \forall x\in C := \{ x \in \reals^2: x_1 \in [0, \infty), x_2 \in [-1,1] \},\\
	& G(x) := \begin{bmatrix} x_1+ \ball \\ x_2 \end{bmatrix}
	& \forall x\in D := \{ x \in \reals^2: x_1 \in [0, \infty), x_2 = 0 \}.
	\end{align*}
}
Let $K = C$ and note the following properties of maximal solutions to $\HS$:
\begin{itemize}[leftmargin = 4mm]
	\item For some $x \in K$, there exists $\phi = (\phi_1, \phi_2) \in \sol_\HS(x)$ with $\rge \phi \subset K$, but is not complete due to $\lim\limits_{t \searrow t^*}\phi_1(t,0) = \infty$ with $t^* <\infty$; for instance, from $x = (0, 1)$, the solution given by $\phi(t,0) = (\tan(t), 1)$ for every $(t,0) \in \dom \phi$ has its $\phi_1$ component escape to infinite as $t$ approaches $t^* = \pi/2$;
	\item From points in $D \subset K$, there exist maximal solutions that leave $K$ and are not complete: such solutions end after a jump because their $x_1$ component is mapped outside of $K$.
\end{itemize}
	\begin{figure}[h]
		\centering
		\includegraphics[width=\figsize]{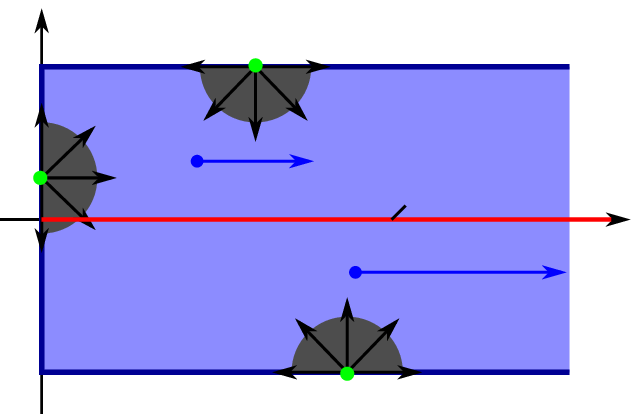}
		\small
		\put(-0.03,.33){$x_1$}
		\put(-1.01,.65){$x_2$}
		\put(-.98,.53){1}
		\put(-.98,.05){-1}
		\put(-.98,.26){0}
		\put(-.7,.17){$C$}
		\put(-.36,.35){$D$}
		\put(-1,.35){$x_v$}
		\put(-.83,.32){$T_{K\cap C}(x_v)$}
		\put(-.62,.58){$x_u$}
		\put(-.5,.5){$T_{K\cap C}(x_u)$}
		\put(-.47,.03){$x_l$}
		\put(-.35,.1){$T_{K\cap C}(x_l)$}
		\caption{Flow and jump sets of the system in Example~\ref{ex:pre}.}
		\label{fig:propwppfi}
	\end{figure}
Thus, we verify weak forward pre-invariance of $K$ by applying Theorem~\ref{thm:wfi}.
The sets $K, C, D$ and the map $F$ satisfy Assumption~\ref{assum:data} by construction and condition \ref{item:wfi1} holds for $\HS$ by definition of $G, D$ and $K$.
Since $L = \emptyset$, condition \ref{item:wfi2} holds because for every $x \in \hC$, $F(x)$ points horizontally, and
\IfConf{
	$$T_{K \cap C} (x) = \begin{cases}
	\reals \times \reals_{\leq 0} & \ifeq x_1 \in (0, \infty), x_2 = 1\\
	\reals \times \reals_{\geq 0} & \ifeq  x_1 \in (0, \infty), x_2 = -1\\
	\reals_{\geq 0} \times \reals_{\leq 0} &\ifeq x = (0,1)\\
	\reals_{\geq 0} \times \reals &\ifeq x_1 = 0,  x_2 \in (-1,1)\\
	\reals_{\geq 0} \times \reals_{\geq 0} & \ifeq x =(0,-1).
	\end{cases}$$
}{
$$T_{K \cap C} (x) = \begin{cases}
\reals \times \reals_{\leq 0} & \ifeq x \in  \{ x \in \reals^2: x_1 \in (0, \infty), x_2 = 1\}\\
\reals \times \reals_{\geq 0} & \ifeq x \in  \{ x \in \reals^2: x_1 \in (0, \infty), x_2 = -1\}\\
\reals_{\geq 0} \times \reals_{\leq 0} &\ifeq x = (0,1)\\
\reals_{\geq 0} \times \reals &\ifeq x \in  \{ x \in \reals^2: x_1 = 0,  x_2 \in (-1,1)\}\\
\reals_{\geq 0} \times \reals_{\geq 0} & \ifeq x =(0,-1).
\end{cases}$$
}
\noindent Tangent cones of $K \cap C$ at points $x_u, x_v$ and $x_l$ of $K$ are shown in Figure~\ref{fig:propwppfi}.

Now, consider the same data but with $G$ replaced by $G'(x) = G(x) \cap (\reals_{\geq 0} \times \reals)$ for each $x \in D$.
The set $K = C$ is forward pre-invariant for this system.
This is because maximal solutions are not able to jump out of $K$ as $G'$ only maps $x_1$ components of solutions to $[0, +\infty)$.
More precisely, the conditions in Theorem~\ref{thm:fi} hold: we have $G'(D \cap K) \subset K$, and Assumption~\ref{assum:data} and condition \ref{item:fi2} hold as discussed above.
%
\hfill $\triangle$
\end{example}

\begin{remark}
	The hybrid inclusions framework allows for an overlap between the flow set $C$ and the jump set $D$.
	As a result, the proposed conditions are not necessary to induce forward invariance properties of sets for $\HS$.
	When existence of nontrivial solutions and completeness are not required for every point in $K$, as in the ``pre'' notions, some of these conditions are necessary.
	In fact, suppose $K,C,D$, and $F$ satisfy Assumption~\ref{assum:data}:
	\begin{itemize}[leftmargin = 4mm]
		\item If $K$ is weakly forward pre-invariant for $\HS$, then for every
		$x \in (K \cap D)\setminus C$, $G(x) \cap K \neq \emptyset.$
		\item If $K$ is forward pre-invariant or forward invariant for $\HS$, then condition \ref{item:fi1} in Theorem~\ref{thm:fi} holds.
		\item If $K$ is weakly forward invariant or forward invariant for $\HS$, then for every $x \in K \setminus D$, $F(x) \cap T_{K \cap C}(x) \neq \emptyset$.\footnote{A similar claim is presented in \cite[Proposition 3.4.1]{aubin.viability} for continuous-time system.}
	\end{itemize}
	Moreover, unlike \cite[Theorem 3]{aubin.hs.viability}, when the flow map $F$ is Marchaud\footnote{A map $F$ is Marchaud on $K \cap C$ when Assumption~\ref{assum:data} holds and $F$ has linear growth on $K \cap C$; see \cite[Definition 2.2.4]{aubin.viability}.} and Lipschitz as defined in Definition~\ref{def:lip}, condition $F(x) \subset T_{K \cap C}(x)$ for every $x \in K \setminus D$ is not necessary as the following example shows.
		Consider $\HS$ in \eqref{eq:H} with data 
		$F(x) =
		\begin{cases}
		1 & \ifeq x > -1\\
		[-1,1] & \ifeq x = -1
		\end{cases}$ 
		for each $x \in C : = [-1, 1]$, $G(x) := \{-1, 0\}$ for each $x \in D: = \{1\}$.
		By inspection, the set $K = C$ is forward invariant for $\HS$ and $F$ is Marchaud and Lipschitz.
		However, at $x = -1 \in K\setminus D$, $F(-1) \supset -1$ but $-1 \notin T_{K \cap C}(-1)$.
\end{remark}

\ConfSp{-4mm}

\subsection{Sufficient Conditions for N$\star$)}\label{sec:nstar}
In Theorem~\ref{thm:wfi} and Theorem~\ref{thm:fi}, item \ref{item:N*} excludes case \ref{item:b.2} in Proposition~\ref{prop:existence}, where solutions escape to infinity in finite time during flows.
In fact, when every solution $\phi$ to $\dot{x} \in F(x)$ with $\phi(0,0) \in K^\star$ does not have a finite escape time, namely, there does not exist $t^* < \infty$ such that $\lim\limits_{t \searrow t^*} |\phi(t)| = \infty$, item \ref{item:N*} holds for $\HS$ and $K^\star$ as defined in Theorem~\ref{thm:wfi} and Theorem~\ref{thm:fi}, respectively.
Although, in principle, such a condition is solution dependent, it can be guaranteed without solving for solutions when $F$ is single valued and globally Lipschitz.
Moreover, we provide several other alternatives in the next result.
\begin{lemma}(sufficient conditions for completeness)\label{lem:Nstar}
	Given $\HS = \data$ and a set $K\subset \reals^n$, suppose $K, C, D,$ and $F$ satisfy Assumption~\ref{assum:data}, \chai{set $D$ is closed} and item
	\ref{item:wfi2} in Theorem~\ref{thm:wfi} holds.\footnote{When \ref{item:Nstar1} holds, condition \ref{item:wfi2} in Theorem~\ref{thm:wfi} is not required.} 
	Condition \ref{item:N*} holds if
	\begin{enumerate}[label = \ref{lem:Nstar}.\arabic*), leftmargin=2.8\parindent]
		\item\label{item:Nstar1} $K^\star$ is compact; or
		\item\label{item:Nstar3} $F$ has linear growth on $K^\star$.
	\end{enumerate}
\end{lemma}
\begin{proof}
	Let $\phi \in \sol_{\HS}(K^\star)$ with $\rge \phi \subset K$ be as described by \ref{item:b.2n} in Proposition~\ref{prop:existence}; namely, $t \mapsto \phi(t,J)$ defined on $I^J$, where $(T,J) = \sup \dom \phi, T+ J < \infty$ and, for some $t^J, I^J = [t_J, T)$.
	Since $t \mapsto \phi(t,J)$ is locally absolutely continuous on $I^J$, $\lim\limits_{t \rightarrow T} \phi(t,J)$ is finite or infinity.
	If it is finite, then $t \mapsto \phi(t,j)$ can be extended to $\overline{I^J}$, which contradicts with \ref{item:b.2n}.
	Then, it has to be that $\lim\limits_{t \rightarrow T} \phi(t,J)$ is infinity.
	When \ref{item:Nstar1} holds, $\lim\limits_{t \rightarrow T} \phi(t,J)$ being infinity is a contradiction since $K^\star$ is compact.
	If \ref{item:Nstar3} holds, \cite[Lemma 1]{aubin.hs.viability}\footnote{The sets $X, K, C$ in \cite[Lemma 1]{aubin.hs.viability} are $C$, $K \cap C$, \chai{$D$ (or $\emptyset$ when applied for Theorem~\ref{thm:fi})}, respectively, in our context.}
	implies that $t \mapsto \phi(t,J)$ is either:
	\begin{itemize}[leftmargin = 4mm]
		\item defined over $[t_J, \infty)$ with values in $K\cap C$;\footnote{\chai{This is the only case that applies for Theorem~\ref{thm:fi}.}} or
		\item defined over $[t_J, T]$ with $T \geq t_J$, $\phi(T, J) \in D$ and $\phi(t, J) \in K \cap C$ for every $t \in [t_J,T]$.
	\end{itemize}
	In either case, we have a contradiction.
\end{proof}


The next example illustrates Theorem~\ref{thm:wfi}, Theorem~\ref{thm:fi} and Lemma~\ref{lem:Nstar}.
\begin{example}(weakly forward invariant set)\label{ex:wFI}
	Consider the hybrid system $\HS = \data$ in $\reals^2$ given by
	\begin{align*}
		&F(x):= (x_2, -x_1) & \forall x \in C;\\
		&G(x):= (-0.9x_1, x_2) &\forall x \in D,
	\end{align*}
	where $C := \{x \in \reals^2: |x| \leq 1, x_2 \geq 0\}$ and
	$D:= \{x \in \reals^2: x_1 \geq -1, x_2 = 0\}.$
	\ConfSp{-3mm}
	\begin{figure}[H]
		\centering
		\includegraphics[width=\figsize]{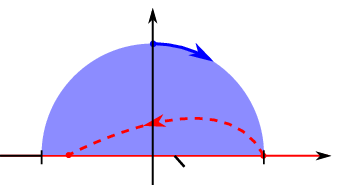}
		\small
		\put(-0.05,.12){$x_1$}
		\put(-.62,.5){$x_2$}
		\put(-.26,.05){1}
		\put(-.93,.05){-1}
		\put(-.6,.05){0}
		\put(-.7,.3){$C$}
		\put(-.45,.03){$D$}
		\caption{Sets and directions of flows/jumps in Example~\ref{ex:wFI}.}
		\label{fig:wFI}
	\end{figure}
	The set $K_1 = \partial C$ is weakly forward invariant for $\HS$ by Theorem~\ref{thm:wfi}.
	More precisely, for every $x \in K_1 \cap D$, $G(x) \in K_1$; and for every $x \in \partial (K_1 \cap C) \setminus (D \cup L) =$ $\{x \in \reals^2: |x| = 1, x_2 > 0\}$, since $\inner{\nabla( x^2_1 + x^2_2 - 1)}{F(x)} = 0$, applying item \ref{item:inner1} in Lemma~\ref{lem:innerTcone}, we have $F(x) \in T_{K_1 \cap C}(x)$.
	In addition, $K_1 \cap C = \partial C$ is compact, which implies condition \ref{item:N*} holds by Lemma~\ref{lem:star}.
	Thus, for every $x \in K_1$, there exists one complete solution that stays in $K_1$.
	For example, for every $x \in [-1,1] \times \{0\}$, there exists one complete solution that is discrete and stays in $K_1$ (from the origin there is also a complete continuous solution that remains at the origin), but also there exist maximal solutions that flow inside $\{x \in \reals^2: |x| < 1 \}$ and leave $K_1$.
	
	\chai{Now consider $K_2 = C$. It is forward invariant for $\HS$ by applying Theorem~\ref{thm:fi}.
	In fact, using the  observations above, item \ref{item:fi2} can be verified via Lemma~\ref{lem:innerTcone} since $\inner{\nabla( x^2_1 + x^2_2 - 1)}{F(x)} = 0$ for every $x \in \partial (K_2 \cap C) \setminus L = \{x \in \reals^2:  |x| = 1, x_2 > 0\} \cup ([-1,0]\times \{0\})$.}
	\hfill $\triangle$
\end{example}

Note that one can replace condition \ref{item:fi2} in Theorem~\ref{thm:fi} by
	\begin{enumerate}[label = \ref{thm:fi}.2$^{'}$), leftmargin= 12mm]
		\item\label{item:fi2'} For every $x \in \partial (K \cap C),$
	\end{enumerate}
	\begin{align}
		& F(x) \subset T_{K \cap C}(x) &\ifeq &x \notin \partial C \cap D\label{eq:fi2''}\\
		&F(x) \cap (T_C(x) \setminus T_{K \cap C}(x)) = \emptyset  &\ifeq &x \in \partial C \cap D. \label{eq:fi2''s}
	\end{align}
Note that assumption \eqref{eq:fi2''s} is important as in some cases, having item \eqref{eq:fi2''} only leads to solutions that escape the set $K$ by flowing as shown in the following example.
Consider the hybrid system $\HS$ on $\reals^2$ with
\begin{align*}
	&F(x) =(x_2, -\gamma) & \forall x &\in C := \{x \in \reals^2: x_1x_2 \geq 0 \}\\
	&G(x) = x & \forall x &\in D := \{x \in \reals^2: x_1 \geq 0, x_2 = 0 \},
\end{align*}
where $\gamma > 0$.
The set $K = \{x \in \reals^2: x_1 \geq 0, x_2 \geq 0 \}$ is weakly forward invariant, and the sets $K, C, D$ and the map $F$ satisfies \eqref{eq:fi2''}.
However, at the origin, we have $F(0) = (0, -\gamma)$ and
\ConfSp{-1mm}
\begin{align*}
	&T_C(0) = (\reals_{\geq 0} \times \reals_{\geq 0} ) \cup (\reals_{\leq 0} \times \reals_{\leq 0}),\\
	&T_{K \cap C}(0) = \reals_{\geq 0} \times \reals_{\geq 0}.
\end{align*}
\ConfSp{-5mm}

\noindent Hence, at the origin, one solution can flow into $C \setminus K$ (the third quadrant) because $F(0) \in T_C(0) \setminus T_{K \cap C}(0)$.

The following example is an application of Theorem~\ref{thm:fi} and Lemma~\ref{lem:Nstar}.
\begin{example}(forward invariant set)\label{ex:oscillator0w}
	Consider the hybrid system given by
	\IfConf{
		\begin{align}\label{eq:ex_zerod}
			\HS \begin{cases}
				x \in C &\dot{x} = F(x) := (- |x_1|x_2, 0)\\
				x \in D & x^+ = G(x) := x,
			\end{cases}
		\end{align}
		}{
	\begin{align}\label{eq:ex_zerod}
		\HS \begin{cases}
			x \in C &\dot{x} = F(x) := \begin{bmatrix}- |x_1|x_2 \\ 0
			\end{bmatrix} \\
			x \in D & x^+ = G(x) := x,
		\end{cases}
	\end{align}
}
	where the flow set is $C = \{x \in \reals^2:  |x| \leq 1, x_1x_2 \geq 0 \}$ and the jump set is $D = \{x \in \reals^2:  |x| \leq 1, x_1x_2 \leq 0 \}$.
	We observe that during flow, solutions evolve continuously within the unit circle centered at the origin; while at jumps, solutions remain at the current location.
	Applying Theorem~\ref{thm:fi}, we show that the set $K_1 = C_1 \cup D_1$ is forward invariant for $\HS,$ where $C_1 = \{x \in \reals^2: x_1 \geq 0, x_2 \geq 0, |x| \leq 1\}$ and $D_1 = \{x \in \reals^2: x_1 \leq 0, x_2 \geq 0, |x| \leq 1\}$.
	By construction, $K_1, C, D$ and $F$ satisfy Assumption~\ref{assum:data}. 
	Condition \ref{item:fi1} holds since $G$ maps the state to its current value.
	Condition \ref{item:fi2} holds since
	\begin{itemize}[leftmargin = 4mm]
		\item for every $x \in \{x \in  \partial C_1: |x| =1 \}$, since $x_1 x_2 \geq 0$,
		\IfConf{$\inner{\nabla (x_1^2 + x_2^2)}{F(x)} = -2|x_1|x_1x_2 \leq 0;$
			}{
		\begin{align*}
		\inner{\nabla (x_1^2 + x_2^2)}{F(x)} = -2|x_1|x_1x_2 \leq 0;
		\end{align*}
	}
		\item for every $x \in \{x \in  \partial C_1: |x| \neq 1 \}$, $F(x) = (0,0)$,
		which leads to $F(x) \in T_{K_1 \cap C}(x)$.
	\end{itemize}
	Finally, applying Lemma~\ref{lem:Nstar}, \ref{item:N*} holds since $K_1 \cap C$ is compact.
	\hfill $\triangle$
\end{example}

	In Example~\ref{ex:oscillator0w}, the set $K_1$ is forward invariant for $\HS$ as shown therein.
	When arbitrarily small disturbances are introduced, solution pairs may escape the set of interest.
	In particular, we revisit Example~\ref{ex:oscillator0w} with disturbances in the next section, and show that $K_1$ is only weakly forward invariant, uniformly in the given disturbances.
	The forthcoming results in Section~\ref{sec:HwSuffConds} are useful to verify such properties of sets for hybrid systems $\Hw$ given as in \eqref{eq:Hw}.

\subsection{Sufficient Conditions for Robust Forward Invariance Properties for $\Hw$}\label{sec:HwSuffConds}
As an extension to the nominal notions, the robust forward invariance notions for $\Hw$ in Definition~\ref{def:rwFI} - \ref{def:rFI} capture four types of forward invariance properties, some of which are uniform over disturbances $w$ for $\Hw$.
In this section, Theorem~\ref{thm:rwFI} and Theorem~\ref{thm:rFI} extend Theorem~\ref{thm:wfi} and Theorem~\ref{thm:fi} to hybrid systems $\Hw$ given in \eqref{eq:Hw}.
These results will be exploited in forward invariance-based control design for hybrid systems (with and without disturbances) in Part II of this work.
Similar to the results in Section~\ref{sec:HSuffConds}, throughout this section, the following version of Assumption~\ref{assum:data} with disturbances is assumed.
\begin{assumption}\label{assum:wdata}
	The sets $K, \Cw$, and $\Dw$ are such that $K\subset \overline{\Pwc{\Cw}}\cup \Pwd{\Dw}$ and that $K \cap \Pwc{\Cw}$ is closed.
	The map $\Fw$ is outer semicontinuous, locally bounded on $(K \times \Wc) \cap \Cw$, and $\Fw(x,w_c)$ is nonempty and convex for every $(x,w_c) \in (K \times \Wc) \cap \Cw$.
	For every $x \in \Pwc{\Cw}, 0 \in \wxc(x)$.
\end{assumption}

	Assumption~\ref{assum:wdata} guarantees that all points in the set to render invariant, namely, $K$, are either in the projections to the state space of $\Cw$ and $\Dw$, which is necessary for solutions from $K$ to exist.
	The closedness of the set $K \cap \Pwc{\Cw}$ and the regularity properties of $\Fw$ are required to obtain conditions in terms of the tangent cone; see, also, \cite[Proposition 6.10]{65}.
	The assumption of $0 \in \wxc(x)$ for every $x \in \Pwc{\Cw}$ usually holds for free since systems with disturbances, such as $\Hw$, typically reduce to the nominal system, in our case $\HS$, when the disturbances vanish.
	A similar property could be enforced for the disturbance $w_d$, but such an assumption is not needed in our results.

Next, we propose sufficient conditions to guarantee robust weak forward pre-invariance and robust weak forward invariance of a set for $\Hw$.
\begin{theorem}(sufficient conditions for robust weak forward (pre-) invariance of a set)\label{thm:rwFI}
	Given $\Hw = \dataw$ as in \eqref{eq:Hw} and a set $K\subset \reals^n$, suppose $\Cw, \Fw, \Dw$ and $K$ satisfy Assumption~\ref{assum:wdata}.
	The set $K$ is robustly weakly forward pre-invariant for $\Hw$ if the following conditions hold:
	\begin{enumerate}[label = \ref{thm:rwFI}.\arabic*), leftmargin=3.2\parindent]
		\item\label{item:rwFI1} For every $x \in K\cap\Pwd{\Dw}$, $\exists w_d \in \wxd(x)$ such that $\Gw(x, w_d) \cap K \neq \emptyset$;
		\item\label{item:rwFI2} For every $x \in \chai{\Pwc{\hCw}} \setminus \Pwd{\Dw}$, $\Fw(x, 0) \cap T_{K \cap \Pwc{\Cw}}(x)$ $\neq \emptyset$;
	\end{enumerate}
	where $\chai{\hCw} : = ((\partial (K\cap\Pwc{\Cw}) \times \Wc) \cap \Cw) \setminus \Lw$ and $\Lw : = \{(x, w_c) \in \Cw: x \in \partial \Pwc{\Cw}, \Fw(x,w_c) \cap T_{\overline{\Pwc{\Cw}}}(x) = \emptyset \}$.
	Moreover, $K$ is robustly weakly forward invariant for $\Hw$ if, in addition, \chai{$K \cap \Pwc{\Lw} \subset \Pwd{\Dw}$ and, with $\widetilde{K}^\star = ((K \setminus \Pwd{\Dw}) \times  \Wc)\cap \Cw$,}
	\begin{enumerate}[label = $\star$), leftmargin=6mm]
		\item\label{item:*} For every $(\phi, w)\in \sol_{\Hw}(\chai{\Pwc{\widetilde{K}^\star}})$ with $\rge \phi \subset K$, case \ref{item:b.2} in Proposition~\ref{prop:wexistence} does not hold.
	\end{enumerate}
\end{theorem}
	
\begin{proof}
	Given $\Cw, \Fw, \Dw$ and $K$ satisfying Assumption~\ref{assum:wdata}, zero disturbance is always admissible to $\Hw$ during continuous evolution of solution pairs.
	We define a restriction of $\Hw$ by $K$ with zero disturbance during flows as follows:
	$\wHS\w = (\wC, \wF, \wD\w, \Gw)$,
	where $\wC := K \cap \Pwc{\Cw}$, $\wF(x) = \Fw(x, 0)$ for every $x \in \Pwc{\Cw}$ and $\wD\w := (K \times \Wd)\cap \Dw$.
	Since $K \subset \overline{\Pwc{\Cw}} \cup \Pwd{\Dw}$, by Definition~\ref{def:HwSol}, there exists a solution pair to $\wHS\w$ from every $\xi \in K$.
	Let $K_1 = \Pwd{\wD\w}$,
	\chai{$K_2 = K \setminus (\Pwd{\wD\w} \cup \Pwc{\Lw})$ and $K_3 = K \setminus (K_1 \cup K_2)$.
	By definition, every $\xi \in K_3$ is such that $\xi \in \Pwc{\Lw}\setminus \Pwd{\wD\w}$ and $\wF(\xi)\cap T_{\overline{\Pwc{\Cw}}}(\xi) =\emptyset$.
	Then, item (a) in \cite[Lemma 5.26]{65} and Definition~\ref{def:solution} imply there is only trivial solution from $\xi$ to $\wHS\w$, in which case we have $\rge \phi \subset K$.
	Otherwise, in the case where $\phi(0,0) \in K_1 \cup K_2$, we show there exists $(\phi, w) \in \sol_{\wHS\w}$ that is nontrivial and it has $\rge \phi \subset K$ when \ref{item:rwFI1} and \ref{item:rwFI2} hold true.
	To this end, we construct a nontrivial solution pair from every $\xi \in K_1 \cup K_2$.}
	Since $K_1$ and $K_2$ are disjoint sets, we have following two cases:
	\begin{enumerate}[label = \roman*), leftmargin = 4mm]
		\item\label{item:rwFIi} when $\xi \in K_1$: since $K_1 \subset \Pwd{\Dw}$, a jump is possible from every $\xi \in K_1$, i.e., from every $(\xi, w_d) \in \widetilde{D}\w$.
		Let $\phi_a(0,0) = \xi$.
		By condition \ref{item:rwFI1}, there exists $\widetilde{w}_d \in \wxd(\xi), \phi_a(0,1) \in \Gw(\xi, \widetilde{w}_d)$, such that $\phi_a(0,1) \in K$.
		\item\label{item:rwFIii} when $\xi \in K_2$:
		since $K \subset \overline{\Pwc{\Cw}} \cup \Pwd{\Dw}$, $\xi \in \overline{\Pwc{\Cw}}\setminus \Pwd{\Dw}$ and solution pairs can only evolve by flowing from $\xi$.
		Conditions enforced by Assumption~\ref{assum:wdata} imply that $\wC$ is closed, $\wF$ is outer semicontinuous, locally bounded and convex valued on $\wC$.
		\chai{Since $T_{\wC}(x) = \reals^n$ for every $x \in (\inter \wC)\setminus (\Pwd{\wD\w} \cup \Pwc{\Lw})$, item \ref{item:rwFI2} implies that $\widetilde{F}(x) \cap T_{\widetilde{C}}(x) \neq \emptyset$ for every $x \in K_2$.
		Then, by an application of \cite[Proposition 6.10]{65}, there exists a nontrivial solution $\phi_b$ to $\wHS\w$ from every $\xi \in K_2$.
		By item \ref{item:S1} in Definition~\ref{def:solution}, such a nontrivial solution $\phi_b$ is absolutely continuous on $[0,\varepsilon]$, for some $ \varepsilon > 0$, with $\phi_b(0) = \xi$, $\dot{\phi}_b(t) \in \wF(\phi_b(t))$ for almost all $t \in [0,\varepsilon]$ and $\phi_b(t) \in \wC$ for all $t \in (0,\varepsilon]$.
		By closedness of $\wC$, we have $\phi_b(t,0) \in K$ for every $t \in [0,\varepsilon]$.}
	\end{enumerate}
	The above shows that from every point in $K_1$, solution pairs to $\wHS\w$ can be extended via jumps with the state component staying within $K$ using the construction in case \ref{item:rwFIi}.
	While from points in $K_2$, solution pairs can be extended using the construction in case \ref{item:rwFIii} with the state component staying within $K$.
	As a consequence, from every point in $K$, there exists at least one $(\widetilde{\phi}, \widetilde{w}) \in \sol_{\wHS\w}$ with $\rge\widetilde{\phi} \subset K$.
	
	Next, we prove that each such $(\widetilde{\phi}, \widetilde{w})$ is also maximal to $\Hw$.\footnote{During flows, we have $(\widetilde{\phi}, 0)$.}
	If $(\widetilde{\phi}, \widetilde{w})$ is complete, then it is already maximal and a solution pair to $\Hw$.
	Consider the case that $(\widetilde{\phi}, \widetilde{w})$ is not complete.
	Proceeding by contradiction, suppose $(\widetilde{\phi}, \widetilde{w})$ is not maximal for $\Hw$, meaning that there exists $(\phi, w)$ such that $\phi(t,j) =   \widetilde{\phi}(t,j)$ and $w(t,j) = \widetilde{w}(t,j)$ for every $(t,j)\in \dom \widetilde{\phi}$ and $\dom \phi \setminus \dom \widetilde{\phi} \neq \emptyset$.
	Let $(T,J) = \sup \dom \widetilde{\phi}$.
	If $(T,J) \in \dom \widetilde{\phi}$, then, $\widetilde{\phi}(T,J) \in K$ and we have the two following cases:
	\chai{
	\begin{itemize}[leftmargin=4mm]
		\item $\widetilde{\phi}(T,J) \in K_1 \cup K_2$, \ref{item:rwFI1} and closeness of $\wC$ imply that, using the arguments in \ref{item:rwFIi}  and \ref{item:rwFIii}  above, it is possible for $\phi$ to satisfy $\phi(t,j) \in K$ for some $(t,j) \in \dom \phi \setminus \dom \widetilde{\phi}$.
		By definition of solution pairs, this contradicts with maximality of $(\widetilde{\phi}, \widetilde{w})$ for $\wHS\w$.
		\item $\widetilde{\phi}(T,J) \in K_3$, by definition of $\Lw$, $\Fw(\widetilde{\phi}(T,J), w_c) \cap T_{\overline{\Pwc{\Cw}}}(\widetilde{\phi}(T,J)) =\emptyset$ for every $w_c \in \wxc{(\widetilde{\phi}(T,J))}$.
		Hence, $\sup \dom \phi = (T,J)$, which contradicts with the assumption $\dom \phi \setminus \dom \widetilde{\phi} \neq \emptyset$.
	\end{itemize}
}
	If $(T,J) \notin \dom \widetilde{\phi}$, according to Proposition~\ref{prop:wexistence}, only \ref{item:b.2} holds.\footnote{Case \ref{item:a} does not hold due to $(\widetilde{\phi}, \widetilde{w})$ not being complete, while \ref{item:b.1} and \ref{item:c} do not hold because $(T,J) \not\in \dom \widetilde{\phi}$.}
	In such a case, there is no function $z : \overline{I^J} \rightarrow \reals^n$ satisfying the conditions in \ref{item:b.2} of Proposition~\ref{prop:wexistence}, which are needed to have a $(\phi, w)$ such that $\dom \phi\setminus \dom \widetilde{\phi} \neq \emptyset$.
	Thus, $K$ is robustly weakly forward pre-invariant for $\Hw$.
	
	The last claim requires to show that among these maximal solution pairs to $\Hw$ that stay in $K$ for all future time, there exist one complete solution pair from every point in $K$ when, in addition, $(K \cap \Pwc{\Lw}) \subset \Pwd{\Dw}$ and item \ref{item:*} hold.
	To this end, \chai{first, note that the existence of a nontrivial solution pair to $\Hw$ from every $x \in K$ follows from $(K \cap \Pwc{\Lw}) \subset \Pwd{\Dw}$, which implies $K_3 = \emptyset.$}
	Then, we apply Proposition~\ref{prop:wexistence} to complete the proof.
	Proceeding by contradiction, given any $\xi \in K$, suppose every $(\phi^*,w^*)\in \sol_{\Hw}(\xi)$ is not complete, i.e., $(T,J) = \sup \dom \phi^*, T + J < \infty$, and case \ref{item:a} in Proposition~\ref{prop:wexistence} does not hold.
	Such a solution pair $(\phi^*, w^*)$ is not as described in case \ref{item:b.1.1} in Proposition~\ref{prop:wexistence} due to the closeness of $K \cap \Pwc{\Cw}$.
	Case \ref{item:c.1} does not hold for $(\phi^*, w^*)$ either, since $\rge \phi^* \subset K$ and $K \subset \overline{\Pwc{\Cw}} \cup \Pwd{\Dw}$.
	Thus, by Proposition~\ref{prop:wexistence}, $(\phi^*, w^*)$ can only end as described by case \ref{item:b.1.2}, \ref{item:b.2} or \ref{item:c.2}.
	\begin{itemize}[leftmargin= 4mm]
		\item The solution pair ends because the functions described in case \ref{item:b.1.2} or \ref{item:c.2} of Proposition~\ref{prop:wexistence}, i.e., $\widetilde{z}$ does not exist for $(\phi^*(T,J), w^*(T,J))$.
		However, using the same argument in item \ref{item:rwFIii} above with $\widetilde{w}_c \equiv 0$, for every $(x, 0) \in K_1 \times 0$ there exists $\widetilde{z}$ such that \ref{item:b.1.2} holds, which leads to a contradiction.
		\item \chai{If $(\phi^*, w^*)$ is as described by case \ref{item:b.2}, $\phi^*(0,0) \notin \Pwc{\widetilde{K}^\star}$ by assumption \ref{item:*}.
		More precisely, $\phi^*(0,0) \in K_1$, hence, the solution pair can be extended following the same construction in \ref{item:rwFIi} above, which contradicts with the maximality of $(\phi^*, w^*)$.}
	\end{itemize}
\vspace{-5mm}
\end{proof}

Condition \ref{item:rwFI1} in Theorem~\ref{thm:rwFI} guarantees that for every $x\in K \cap \Pwd{\Dw}$ \chai{such that} there exists $w_d \in \wxd(x)$, the jump map contains an element that also belongs to $K$.
Under the stated assumptions, condition \ref{item:rwFI2} implies the satisfaction of \ref{item:VCw} with zero disturbance $w_c$, which suffices for the purpose of Theorem~\ref{thm:rwFI} as it is about weak forward invariance notions.
While involving the tangent cone of $K \cap \Pwc{\Cw}$ in condition \ref{item:rwFI2} is natural, such solution property is more than needed for robust weak forward pre-invariance of $K$ as defined in Definition~\ref{def:rwFI}.
\chai{
	Similarly to Lemma~\ref{lem:Nstar}, solution-independent conditions that imply \ref{item:*} are derived for the disturbance case.
	\begin{lemma}(sufficient conditions for completeness)\label{lem:star}
		Given $\Hw = \dataw$ and a set $K\subset \reals^n$, suppose $K, \Cw, \Dw,$ and $\Fw$ satisfy Assumption~\ref{assum:wdata}, set $\Pwd{\Dw}$ is closed and item \ref{item:rwFI2} in Theorem~\ref{thm:rwFI} holds.
		Condition \ref{item:*} in Theorem~\ref{thm:rwFI} holds if
		\begin{enumerate}[label = \ref{lem:star}.\arabic*), leftmargin=3\parindent]
			\item\label{item:star1} $\widetilde{K}^\star$ is compact; or
			\item\label{item:star3} $\Fw$ has linear growth on $\widetilde{K}^\star$.
		\end{enumerate}
	\end{lemma}
}

The following example illustrates Theorem~\ref{thm:rwFI}.
\begin{example}[robustly weakly forward invariant set] \label{ex:oscillator}
	Consider a variation of hybrid system $\HS$ in Example~\ref{ex:oscillator0w} with disturbances given by
	\IfConf{
	\begin{align}\label{eq:ex_d}
		\Hw \begin{cases}
			(x, w_c) \in \Cw &
			\dot{x} = \Fw(x, w_c) := (-|x_1| x_2, w_c |x_1|x_1)\\
			(x, w_d) \in \Dw & x^+ \in \Gw(x, w_d) :=\\
			&\qquad  \{R(\theta)x : \theta \in [w_d, -w_d] \},
		\end{cases}
	\end{align}

\vspace{-3mm}
	}{
	\begin{align}\label{eq:ex_d}
		\Hw \begin{cases}
			(x, w_c) \in \Cw &
			\dot{x} = \Fw(x, w_c) := |x_1|\begin{bmatrix}- x_2 \\ w_c x_1
			\end{bmatrix} \\
			(x, w_d) \in \Dw & x^+ \in \Gw(x, w_d) := \{R(\theta)x : \theta \in [w_d, -w_d] \},
			\end{cases}
	\end{align}
}
	\noindent where $R(\theta) : =
	\begin{bmatrix}
	\cos \theta & \sin \theta\\
	-\sin \theta & \cos \theta
	\end{bmatrix}$ is a rotation matrix, 
	$\Cw:= \{(x, w_c) \in \reals^2 \times \reals: 0 \leq w_c \leq |x| \leq 1, x_1x_2 \geq 0\}$, and $\Dw:= \{ (x, w_d) \in \reals^2 \times \reals: x_1x_2 \leq 0, |x| \leq 1, -\frac{\pi}{4} \leq w_d \leq 0\}$.
	As shown in Figure~\ref{fig:oscillator}, the projections of $\Cw$ and $\Dw$ onto $\reals^2$ are given by $\Pwc{\Cw} = C_1 \cup C_2$ on $\reals^2$ with $C_2 = \{x \in \reals^2: x_1 \leq 0, x_2 \leq 0, |x| \leq 1\}$ and by $\Pwd{\Dw} = D_1 \cup D_2$ with $D_2 = \{x \in \reals^2: x_1 \geq 0, x_2 \leq 0, |x| \leq 1\}$, respectively.\footnote{We use the same definitions for $K_1$, $C_1$, and $D_1$ as in Example~\ref{ex:oscillator0w}.}
	Based on provided dynamics, solutions travel counter-clockwise during flows, while they either rotate clockwise or counter-clockwise during jumps.
	As a result, solutions can evolve in any of the four quadrants in $\reals^2$, either by flow or jump.
	\IfConf{
	\vspace{-4mm}
		\begin{figure}[H]
			\centering
			\setlength{\unitlength}{.3\textwidth}
			\includegraphics[width = \unitlength]{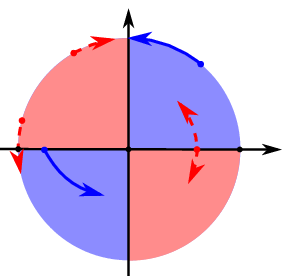}
			\put(-0.1,0.5){$x_1$}
			\put(-0.5,0.9){$x_2$}
			\put(-0.47,0.7){$C_1$}
			\put(-0.7,0.12){$C_2$}
			\put(-0.7,0.7){$D_1$}
			\put(-0.47,0.12){$D_2$}
			\put(-0.52, 0.39){0}
			\put(-0.13, 0.39){1}
			\put(-1.02, 0.39){-1}
			\put(-1, 0.55){$\xi_1$}
			\put(-0.84, 0.49){$\xi_2$}
			\caption{Projection onto the state space of flow and jump sets of the system in Example~\ref{ex:oscillator}. The blue solid arrows indicate possible hybrid arcs during flow, while the red dashed arrows indicate possible hybrid arcs during jumps.}
			\label{fig:oscillator}
		\end{figure}
		
\vspace{-3mm}
	}{
	\begin{figure}[h]
		\centering
		\includegraphics[width = \figsize]{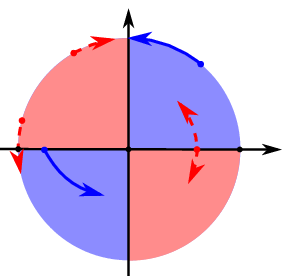}
		\put(-0.1,0.5){$x_1$}
		\put(-0.5,0.9){$x_2$}
		\put(-0.47,0.7){$C_1$}
		\put(-0.7,0.12){$C_2$}
		\put(-0.7,0.7){$D_1$}
		\put(-0.47,0.12){$D_2$}
		\put(-0.52, 0.39){0}
		\put(-0.13, 0.39){1}
		\put(-1.02, 0.39){-1}
		\put(-1, 0.55){$\xi_1$}
		\put(-0.84, 0.49){$\xi_2$}
		\caption{Projection onto the state space of flow and jump sets of the system in Example~\ref{ex:oscillator}. The blue solid arrows indicate possible hybrid arcs during flow, while the red dashed arrows indicate possible hybrid arcs during jumps.}
		\label{fig:oscillator}
	\end{figure}
			}
	First, we apply Theorem~\ref{thm:rwFI} to conclude robust weak forward invariance of the set $K_1 = C_1 \cup D_1$ for $\Hw$.
	Assumption~\ref{assum:wdata} holds for $K_1, \Cw, \Dw$ and $\Fw$ by construction.
	Since the set $\Lw$ is empty, condition \ref{item:rwFI1} holds since for every $(x, w_d) \in (K_1 \times \Wd)\cap \Dw$, the selection $x^+ = R(0)x$ always results in $x^+ \in K_1$.
	Condition \ref{item:rwFI2} holds since, applying item \ref{item:inner1} in Lemma~\ref{lem:innerTcone}, for every $x \in \partial C_1 \setminus \Pwd{\Dw}$, since $x_1x_2 \leq 0$, we have
		\IfConf{
		$$\inner{\nabla (x_1^2 + x_2^2)}{\Fw(x, 0)} = -2x_1x_2|x_1| \leq 0.$$
	}{
		\begin{align*}
			\inner{\nabla (x_1^2 + x_2^2 - 1)}{\Fw(x, 0)} &= 2x_1(-x_2|x_1|) + 2x_2(w_cx_1|x_1|)\\
			& = -2x_1x_2|x_1| \leq 0.
		\end{align*}
	}
	Then, the robust weak forward invariance of $K_1$ follows from \ref{item:star3} in Lemma~\ref{lem:star} and Theorem~\ref{thm:rwFI}.
	Note that the property is weak due to the following observations:
	\begin{itemize}[leftmargin = 4mm]
		\item Because of the set-valuedness of the map $\Gw$, there exists a solution pair from a point $\xi_1 \in D_1$ that jumps to a point in $C_2$ that is not in $K_1$, as depicted in Figure~\ref{fig:oscillator}.
		On the other hand, from the same point $\xi_1$, there exists a solution pair that keeps jumping from and to $\xi_1$, and stays within $D_1 \subset K_1$;
		\item Because of the overlap between $\Pwc{\Cw}$ and $\Pwd{\Dw}$, there exists a solution pair that starts from a point $\xi_2 \in D_1$ and flows to a point in $C_2$ that is not in $K_1$, as depicted in Figure~\ref{fig:oscillator}.
		On the other hand, the solution pair that jumps from and to $\xi_2$ from $\xi_2$ stays within $D_1 \subset K_1$. \hfill $\triangle$
	\end{itemize}
\end{example}

To derive a set of sufficient conditions guaranteeing the stronger robust forward invariance property of $K$, i.e., every solution pair to $\Hw$ is such that its state component stays within the set $K$, when starting from $K$, we require the disturbances $w$ and the set $K$ to satisfy the following assumption.
\begin{assumption}\label{assum:wbound}
	For every $\xi \in (\partial K) \cap \Pwc{\Cw}$, there exists a neighborhood $U$ of $\xi$ such that $\wxc(x) \subset \wxc(\xi)$ for every $x \in U \cap \Pwc{\Cw}$.
\end{assumption}
The next result provides conditions implying robust forward pre-invariance and robust forward invariance of a set for $\Hw$.
\begin{theorem}(sufficient conditions for robust forward \ (pre-) invariance of a set)\label{thm:rFI}
	Given $\Hw = \dataw$ as in \eqref{eq:Hw} and a set $K\subset \reals^n$, suppose $\Cw, \Fw, \Dw$ and $K$ satisfy Assumption~\ref{assum:wdata}.
	Furthermore, suppose the mapping $x \mapsto \Fw(x, w_c)$ is locally Lipschitz uniformly in $w_c$ on $\Cw$.
	The set $K$ is robustly forward pre-invariant for $\Hw$ if the following conditions hold:
	\begin{enumerate}[label = \ref{thm:rFI}.\arabic*), leftmargin=3.2\parindent]
		\item\label{item:rFI1} For every $(x, w_d) \in (K \times \Wd)\cap \Dw, \Gw(x, w_d) \subset K$;
		\item\label{item:rFI2} For every $(x, w_c) \in \chai{\hCw}$, $\Fw(x,w_c) \subset T_{K\cap \Pwc{\Cw}}(x)$,
	\end{enumerate}
	where $\hCw$ and $\Lw$ be given as in Theorem~\ref{thm:rwFI}.
	Moreover, $K$ is robustly forward invariant for $\Hw$ if, in addition, \chai{$K \cap \Pwc{\Lw} \subset \Pwd{\Dw}$ and, with $\widetilde{K}^\star = (K \cap \Pwc{\Cw})\times  \Wc)\cap \Cw$,} condition \ref{item:*} in Theorem~\ref{thm:rwFI} holds.
\end{theorem}

\begin{proof}
	Since condition \ref{item:rFI1} and \ref{item:rFI2} imply condition \ref{item:rwFI1} and \ref{item:rwFI2}, respectively, under given conditions, which include the fact that $\Cw, \Fw, \Dw$ and $K$ satisfy Assumption~\ref{assum:wdata}, the set $K$ is robustly weakly forward pre-invariant for $\Hw$ by Theorem~\ref{thm:rwFI}.

	Now we show that every $(\phi, w) \in \sol_{\Hw}(K)$ has $\rge \phi \subset K$.
	Proceeding by contradiction, suppose there exists a solution pair $(\phi, w) \in \sol_{\Hw}(K)$ such that $\rge \phi \setminus K \neq \emptyset$.
	Then, there exists $(t^*,j^*) \in \dom \phi$ such that $\phi(t^*,j^*) \not \in K$, i.e., $\phi$ eventually leaves $K$ in finite hybrid time.\footnote{Note that when $\rge \phi \subset K$ and $\lim\limits_{t+j \to \sup_t \dom \phi + \sup_j \dom \phi} \phi(t,j) = \infty$ (that is, $\phi$ stays in $K$ but escapes to infinity, potentially in finite hybrid time) corresponds to a solution that satisfies the definition of forward invariance for $K$.}
	Then, we have the two following cases:
	\begin{enumerate}[label = \roman*), leftmargin = 4mm]
		\item\label{item:rFIi} In the case that $\phi$ ``left $K$ by jumping," namely, $\phi(t,j) \in K$ for all $(t,j) \in \dom\phi$ with $t+j < t^* +j^*$, $(\phi(t^*,j^*-1), w_d) \in \Dw$ with $\phi(t^*,j^*) \notin K$ for some $w_d \in \wxd(\phi(t^*,j^*-1))$.
		This contradicts item \ref{item:rFI1}.
		More precisely, since $\phi(t^*,j^*-1) \in K \cap \Pwd{\Dw}$, item \ref{item:rFI1} implies that $\phi(t^*, j^*) \in \Gw(\phi(t^*,j^*-1), w_d(t^*,j^*-1)) \subset K$ for every $w_d \in \wxd(\phi(t^*, j^*-1))$.
		Thus, $\phi$ did not leave $K$ by jumping.
		Then, it must be the case that $\phi$ left $K$ by flowing, which is treated in the next item.
		\item\label{item:rFIii} In the case that $\phi$ ``left $K$ by flowing," namely, there exists a hybrid time instant $(\tau^*,j^*) \in \dom \phi$ such that $\phi(t,j^*) \in \overline{\Pwc{\Cw}}\setminus K$ for all $t \in (\tau^*,t^*]$ and $t^* - \tau^*$ is arbitrarily small and positive.
		Moreover, by closedness of $K\cap \Pwc{\Cw}$, we suppose that $\phi(\tau^*,j^*) \in (\partial K) \cap \Pwc{\Cw}$.\footnote{By definition of solution pair, it is the case that $\phi$ left $K\cap \Pwc{\Cw}$ and entered $\overline{\Pwc{\Cw}}\setminus K$ passing through $(\partial K) \cap \Pwc{\Cw}$.}
		Let $t \mapsto \chi(t) \in K \cap \Pwc{\Cw}$ be such that for every $t \in [\tau^*, t^*]$ 
		\begin{align}\label{eq:chi}
			|z(t)|_{K\cap \Pwc{\Cw}} = |z(t)-\chi(t)|,
		\end{align}
		where $z(t) = \phi(t,j^*)$ for all $t \in [\tau^*, t^*]$.
		Such points exist because of the closedness of $K \cap \Pwc{\Cw}$.
		By definition of solution pairs to $\Hw$, the function $t\mapsto |z(t)|_{K\cap \Pwc{\Cw}}$ is absolutely continuous.
		Thus, for almost every $t \in  [\tau^*, t^*]$, $\frac{d}{dt}|z(t)|_{K\cap \Pwc{\Cw}}$ exists and equals to the Dini derivative of $|z(t)|_{K\cap \Pwc{\Cw}}$.
		Let $t$ be such that both $\frac{d}{dt}|z(t)|_{K\cap \Pwc{\Cw}}$ and $\dot{z}(t)$ exist.
		We have
		\begin{align*}
			&\frac{d}{dt}|z(t)|_{K\cap \Pwc{\Cw}} \\
			& \ = \liminf\limits_{h \searrow 0} \frac{|z(t)+h\dot{z}(t)|_{K\cap \Pwc{\Cw}} - |z(t)|_{K\cap \Pwc{\Cw}} }{h},
		\end{align*}
		which, by definition of $\chi(t)$ and \eqref{eq:chi}, satisfies
\IfConf{
		\begin{align*}
			&\frac{|z(t) + h\dot{z}(t)|_{K\cap \Pwc{\Cw}} - |z(t)|_{K\cap \Pwc{\Cw}}}{h} \\
			& \ \leq \frac{|z(t) \scriptsize{-}\chi(t)| \scriptsize{+}|\chi(t)\scriptsize{+}h\dot{z}(t)|_{K\cap \Pwc{\Cw}} \scriptsize{-}|z(t)|_{K\cap \Pwc{\Cw}}}{h} \\ 
			& \ = \frac{|\chi(t)+h\dot{z}(t)|_{K\cap \Pwc{\Cw}}}{h}\\
			&\ \leq \frac{|\chi(t)+h\omega|_{K\cap \Pwc{\Cw}}}{h} + |\dot{z}(t) - \omega|,
		\end{align*}
}{
		\begin{align*}
			&\!\!\!\!\!\!\!\!\! \frac{|z(t) + h\dot{z}(t)|_{K\cap \Pwc{\Cw}} - |z(t)|_{K\cap \Pwc{\Cw}}}{h} \\
			& \leq \frac{|z(t) - \chi(t)| + |\chi(t)+h\dot{z}(t)|_{K\cap \Pwc{\Cw}} - |z(t)|_{K\cap \Pwc{\Cw}}}{h} =  \frac{|\chi(t)+h\dot{z}(t)|_{K\cap \Pwc{\Cw}}}{h}\\
			&\leq \frac{|\chi(t)+h\omega|_{K\cap \Pwc{\Cw}}}{h} + |\dot{z}(t) - \omega|,
		\end{align*}
}
		for every $\omega \in T_{K\cap \Pwc{\Cw}}(\chi(t)).$
		Moreover, for every such $\omega$,
	$$\liminf\limits_{h \searrow 0} \frac{|\chi(t)+h\omega|_{K\cap \Pwc{\Cw}}}{h} = 0$$
	by definition of the tangent cone in \eqref{eq:tangent}.
	Hence, we have
\IfConf{
		\begin{align*}
			\frac{d}{dt}&|z(t)|_{K\cap \Pwc{\Cw}} \\
			& \leq \liminf\limits_{h \searrow 0} \frac{|\chi(t)+h\omega|_{K\cap \Pwc{\Cw}}}{h} + |\dot{z}(t) - \omega| \\
			& = |\dot{z}(t) - \omega|.
		\end{align*}
	}{
		\begin{align*}
			\frac{d}{dt}|z(t)|_{K\cap \Pwc{\Cw}}
			\leq \liminf\limits_{h \searrow 0} \frac{|\chi(t)+h\omega|_{K\cap \Pwc{\Cw}}}{h} + |\dot{z}(t) - \omega|
			= |\dot{z}(t) - \omega|.
		\end{align*}
}
		Thus, for almost every $t \in  [\tau^*, t^*]$,
		$$\frac{d}{dt}|z(t)|_{K\cap \Pwc{\Cw}} \leq |\dot{z}(t)|_{T_{K\cap \Pwc{\Cw}}(\chi(t))}.$$
		Since $K \cap \Pwc{\Cw}$ is closed, by definition, $\chi(t) \in K \cap \Pwc{\Cw}$ for every $t \in [\tau^*, t^*]$.
		Condition \ref{item:rFI2} implies that for almost all $t \in  [\tau^*, t^*]$, and every $w \in \wxc(\chi(t))$, we have
		\begin{align}\label{eq:dini}
			\frac{d}{dt}|z(t)|_{K\cap \Pwc{\Cw}} 
			&\leq |\dot{z}(t)|_{T_{K\cap \Pwc{\Cw}}(\chi(t))}\\
			&\leq |\dot{z}(t)|_{\Fw(\chi(t), w)}.
		\end{align}
		Since $t^* - \tau^*$ is positive and can be arbitrarily small, it is always possible to construct a neighborhood of $\chi(t)$ for every $t \in [\tau^*, t^*]$, denoted $U$, with $z(t) \in U$, and it is such that $\wxc(z(t)) \subset \wxc(\chi(t))$ by Assumption~\ref{assum:wbound}.
		Then, because of that and the fact that the mapping $x \mapsto \Fw(x, w_c)$ is locally Lipschitz uniformly in $w_c$ on $\Cw$, we can construct a neighborhood $U'$ of $z(t)$ such that $\chi(t) \in U'$ for every $t \in  [\tau^*, t^*]$ and for which there exists a constant $\lambda > 0$ satisfying
		$$\Fw(z(t), w_c) \subset \Fw(\chi(t), w_c) + \lambda|z(t)- \chi(t)|\ball$$
		for every $t \in [\tau^*, t^*]$ and every $w_c \in \wxc(z(t))$.
		Hence, for every $t \in [\tau^*, t^*]$, every $w_c \in \wxc(z(t))$, and every $\eta \in \Fw(z(t), w_c)$,
		$$|\eta|_{\Fw(\chi(t), w_c)} \leq \lambda|z(t)- \chi(t)|.$$
		Moreover, since $\dot{z}(t) \in \Fw(z(t), w_c)$, for every $w_c \in \wxc(z(t))$, together with \eqref{eq:dini} and \eqref{eq:chi}, we have that
		\begin{align*}
			\frac{d}{dt}|z(t)|_{K\cap \Pwc{\Cw}} &\leq |\dot{z}(t)|_{\Fw(\chi(t), w_c)}\\
			&\leq \lambda|z(t) - \chi(t)| = \lambda|z(t)|_{K\cap \Pwc{\Cw}}.
		\end{align*}
		Then, by the Gronwall Lemma (see \cite[Lemma A.1]{khalil.nonlinear.sys}), for every $t \in  [\tau^*, t^*]$,
		$$|z(t)|_{K \cap \Pwc{\Cw}} = 0.$$
		Since $K \cap \Pwc{\Cw}$ is closed, $\phi(t^*,j^*) \in K \cap \Pwc{\Cw}$, which contradicts the definition of $t^*$.
		Thus, there does not exist maximal solution pair $(\phi, w) \in \sol_{\Hw}(K)$ that eventually leaves $K\cap \Pwc{\Cw}$ by flowing.
	\end{enumerate}
	Thus, the set $K$ is robustly forward pre-invariant for $\Hw$.
	
	\chai{
	Following the proof of Theorem~\ref{thm:rwFI}, when $K \cap \Pwc{\Lw} \subset \Pwd{\Dw}$, with \ref{item:rwFI1} and \ref{item:rwFI2} satisfied, there exists a nontrivial solution pair $(\phi, w)$ with $\phi(0,0) = \xi$ to $\Hw$ from every $\xi \in K$.
	Then,} robust forward invariance of $K$ follows from the addition of condition \ref{item:*}.
	As shown above, every $(\phi, w) \in \sol_{\Hw}(K)$ has $\rge \phi \subset K$, thus, it suffices to show that every maximal solution pair to $\Hw$ is complete.
	We proceed by contradiction. Suppose there exists a maximal solution pair $(\phi^*, w^*) \in \sol_{\Hw}(K)$ that is not complete, and $(T,J) = \sup \dom \phi^*$.
	Because every $(\phi, w) \in \sol_{\Hw}(K)$ has $\rge \phi \subset K$, by an application of Proposition~\ref{prop:wexistence}, $(\phi^*, w^*)$ only satisfies one of the cases described in item \ref{item:b.1.2}, \ref{item:b.2}, and \ref{item:c.2}.
	In particular, condition \ref{item:*} eliminates case \ref{item:b.2} by assumption.
	Then, condition \ref{item:rFI1} and condition \ref{item:rFI2} imply that $(\phi^*, w^*)$ can be extended within $K$ by jumps and flows, respectively.
	More precisely, when $\phi^*(T,J) \in \Pwc{\Cw}$, conditions in Assumption~\ref{assum:wdata} and item \ref{item:rFI2} imply the function $\widetilde{z} : [0, \varepsilon] \rightarrow \reals^n$ as described in \ref{item:VCw} in Proposition~\ref{prop:wexistence} exists with $\widetilde{w}_c(t) = 0$ for every $t \in [0,\varepsilon]$, and such $(\widetilde{z}, \widetilde{w}_c)$ can be used to extend $(\phi^*, w^*)$ to hybrid instant $(T+ \varepsilon, J)$, which contradicts the maximality of $(\phi^*, w^*)$.\footnote{Note that the resulting disturbance will be Lebesgue measurable and locally essentially bounded on interval $I^J$.}
	When $\phi^*(T,J) \in \Pwd{\Dw}$, jumps are always possible by virtue of condition \ref{item:rFI1}.
	Therefore, the set $K$ is robustly forward invariant for $\Hw$.
\end{proof}
%

\begin{remark}
	In comparison to Theorem~\ref{thm:rwFI}, Lipschitzness of the set-valued map $\Fw$ (uniformly in $w$) is assumed.
	Together with Assumption~\ref{assum:wbound}, they are crucial to ensure that every solution pair stays in the designated set during flows.
	Note that Assumption~\ref{assum:wbound} guarantees such property uniformly in $w_c$ (see the proof of Theorem~\ref{thm:rFI} for details).
	We refer readers to the example provided below Theorem 3.1 in \cite{blanchini1999set}, which shows solutions leave a set due to the absence of locally Lipschitzness of the right-hand side of a continuous-time system.
\end{remark}

The following example shows an application of Theorem~\ref{thm:rFI}.
\begin{example}[Example~\ref{ex:oscillator} revisited]\label{ex:reoscillator}
	Consider the hybrid system in Example~\ref{ex:oscillator}.
	We apply Theorem~\ref{thm:rFI} to show the set $K_2 = \Pwc{\Cw} \cup \Pwd{\Dw}$ is robustly forward invariant for $\Hw.$
	Similar to Example~\ref{ex:oscillator}, $\Lw = \emptyset$, Assumption~\ref{assum:wdata} and condition \ref{item:*} hold for $K_2, \Fw, \Cw$ and $\Dw$.
	Moreover, Assumption~\ref{assum:wbound} holds since $w_c \leq |x|$ for every $x \in \Pwc{\Cw}$ and the map $\Fw$ is locally Lipschitz on $\Cw$ by construction. Then, condition \ref{item:rFI1} holds since for every $(x, w_d) \in (K_2 \times \Wd)\cap \Dw$, the map $\Gw$ only ``rotates" the state variable $x$ without changing $|x|$ within the unit circle centered at the origin. Condition \ref{item:rFI2} holds since
	\begin{itemize}[leftmargin = 4mm]
		\item for every $(x, w_c) \in (\partial K_2 \times \Wc)\cap \Cw$, because $0 \leq w_c\leq |x| \leq 1$ and $x_1x_2 \geq 0$, we have
		\begin{align*}
			&\inner{\nabla (x_1^2 + x_2^2)}{\Fw(x, w_c)} \\
			&\quad = 2x_1(-x_2|x_1|) + 2x_2(w_cx_1|x_1|)\\
			&\quad = 2x_1x_2 (w_c - 1)|x_1| \leq 0,
		\end{align*}
		which, applying item \ref{item:inner1} in Lemma~\ref{lem:innerTcone}, implies $\Fw(x, w_c) \in T_{K_2 \cap \Pwc{\Cw}}(x);$
		\item for every $(x, w_c) \in ((\partial (\Pwc{\Cw})\setminus \partial K_2) \times \Wc) \cap \Cw$, \IfConf{the tangent cone $T_{K_2 \cap \Pwc{\Cw}}(x)$ is given by
		$$ \begin{cases}
		\reals_{\geq 0} \times \reals & \ifeq x \in C_1, x_1 = 0, x_2 \notin \{0,1\}\\
		\reals_{\leq 0} \times \reals & \ifeq x \in C_2, x_1 = 0, x_2 \notin \{0,-1\}\\
		\reals \times \reals_{\geq 0} & \ifeq x \in C_1, x_1 \notin \{0,1\}, x_2 = 0\\
		\reals \times \reals_{\leq 0} & \ifeq x \in C_2, x_1 \notin \{0,-1\}, x_2 = 0\\
		\reals_{\geq 0}^2 \cup \reals_{\leq 0}^2 & x = 0,
		\end{cases}$$
	}{
	we have
		$$T_{K_2 \cap \Pwc{\Cw}}(x) =
		\begin{cases}
			\reals_{\geq 0} \times \reals & \ifeq x \in C_1, x_1 = 0, x_2 \notin \{0,1\}\\
			\reals_{\leq 0} \times \reals & \ifeq x \in C_2, x_1 = 0, x_2 \notin \{0,-1\}\\
			\reals \times \reals_{\geq 0} & \ifeq x \in C_1, x_1 \notin \{0,1\}, x_2 = 0\\
			\reals \times \reals_{\leq 0} & \ifeq x \in C_2, x_1 \notin \{0,-1\}, x_2 = 0\\
			\reals_{\geq 0}^2 \cup \reals_{\leq 0}^2 & x = 0,
		\end{cases}$$
}
		which, applying item \ref{item:inner1} in Lemma~\ref{lem:innerTcone}, implies $\Fw(x, w_c) \in T_{K_2 \cap \Pwc{\Cw}}(x)$ holds true by definition of $\Fw$.\footnote{We recall from Example~\ref{ex:oscillator} that $C_1 = \{x \in \reals^2: x_1 \geq 0, x_2 \geq 0, |x| \leq 1\}$ and $C_2 = \{x \in \reals^2: x_1 \leq 0, x_2 \leq 0, |x| \leq 1\}$.}
	\end{itemize}
	Thus, the set $K_2$ is robustly forward invariant for $\Hw.$
	\hfill $\triangle$
	\end{example}

\section{Nominal Forward Invariance of Sublevel Sets of Lyapunov-like Functions}
\label{sec:HwLya}
For many control problems, Lyapunov-like functions $\Ly: \reals^n \rightarrow \reals$ for $\HS$ and $\Hw$ can be obtained via analysis or numerical methods.
For such systems, we can verify the robust and nominal forward invariance of the $\level-$sublevel sets of $\Ly$ by exploiting the nonincreasing property of $\Ly$ along solutions.
In this work, for the nominal case, conditions on the system data, namely $\data$ in Theorem~\ref{thm:wfi} and Theorem~\ref{thm:fi} are explored to guarantee the forward invariance of a subset of its $\level-$sublevel set that is given by
\begin{align}\label{eq:Mr}
	\K= \Ls \cap (C \cup D).
\end{align}
We leave the more generic study of robust forward invariance properties for $\Hw$ via Lyapunov methods for Part II of this work, where the Lyapunov functions are used to select feedback laws that render robust forward invariance of their sublevel sets.

The next result introduces a set of constructive conditions that induce weak forward invariance and forward invariance for $\K$ in \eqref{eq:Mr} for $\HS$.
These conditions ensure that solutions stay within $\K$ and also guarantee existence and completeness of nontrivial solutions from every point in the set $\K$.
For convenience, given a function $\Ly$ and two constants $\level, \rU \in \reals$ with $\level \leq \rU$, we define the set $\Ir : = \{x \in \reals^n: \level \leq \Ly(x) \leq \rU \}.$
\begin{theorem}(weak forward invariance and forward invariance of $\K$)\label{thm:Ly}
	Given a hybrid system $\HS = \data$ as in \eqref{eq:H}, suppose the set $C$ is closed, the map $F : \reals^n \rightrightarrows \reals^n$ is outer semicontinuous and locally bounded, and $F (x)$ is nonempty and convex for all $x \in C$.
	Suppose there exist a constant $\rU \in \reals$ and a function $\Ly:\reals^n \rightarrow \reals$ that is continuously differentiable on an open set containing $L_{\Ly}(\rU)$ such that
	\begin{align}
		&\inner{\nabla \Ly(x)}{\eta} \leq 0 &\forall & x \in \Ir \cap C, \eta \in F(x),\label{eq:lya1}\\
		& \Ly(\eta) \leq \level & \forall & x \in \Ls \cap D, \eta \in G(x), \label{eq:lya2}
	\end{align}
	for some $\level \in (-\infty, \rU)$.
	Moreover, suppose such $\level$ satisfies
	\begin{enumerate}[label = \ref{thm:Ly}.\arabic*), leftmargin=2.8\parindent]
		\item\label{item:Ly1} for every $x \in \Ly^{-1}(\level)$, $\nabla \Ly(x) \neq 0$;
		\item\label{item:Ly2} for every $x \in (\Ls \cap \partial C)\setminus D$, $F(x) \cap T_C(x) \neq \emptyset$;
		\item\label{item:Ly3} for every $x \in (\Ly^{-1}(\level) \cap \partial C)\setminus D,$ the set $C$ is regular at $x$ and $\exists \xi \in F(x) \cap T_C(x)$, $\inner{\nabla \Ly(x)}{\xi} < 0.$
		\item\label{item:Ly4} condition \ref{item:N*} in Theorem~\ref{thm:wfi} holds for $K^\star = \K \cap C$ and $\HS$.
	\end{enumerate}
	Then, for each such $\level \in (-\infty,\rU)$ that defines a nonempty and closed $\K$, we have the following:
	\begin{itemize}
		\item The set $\K$ is weakly forward invariant for $\HS$ if
		\begin{enumerate}[label = \ref{thm:Ly}.\arabic*), resume, leftmargin=1.6\parindent]
			\item\label{item:Ly5} for every $x \in \K \cap D$, $G(x) \cap (C \cup D) \neq \emptyset$;
		\end{enumerate}
		\item The set $\K$ is forward invariant for $\HS$ if
		\begin{enumerate}[label = \ref{thm:Ly}.\arabic*), resume, leftmargin=1.6\parindent]
			\item\label{item:Ly6} $G(\K \cap D) \subset C \cup D$.
		\end{enumerate}
	\end{itemize}
\end{theorem}

\begin{proof}
	Fix $\level < \rU$ that satisfies the conditions in Theorem~\ref{thm:Ly}.
	The sets $K = \K, C, D$ and the map $F$ satisfy Assumption~\ref{assum:data}.
	In fact, since $\K$ is defined as the intersection of an $\level$-sublevel set of $\Ly$ and the union of the flow set and the jump set, $\K$ is a subset of $C \cup D$.
	Closedness of $\K \cap C$ follows from the fact that $C$ is closed and $\Ly$ is continuous.
	The properties of $F$ directly follow from the assumptions.
	Now, we apply Theorem~\ref{thm:wfi} to prove weak forward invariance of the set $\K$.
	
	Since set $L$ in Theorem~\ref{thm:wfi} is empty in this case, we prove that for every $x \in \partial (\K \cap C) \setminus D$,
	\begin{align}\label{eq:intersect}
		F(x) \cap T_{\Ls \cap C}(x) \neq \emptyset.
	\end{align}
	To this end, we need the following properties of the sets $C, \Ls$ and of the map $F$.
	For every $x \in \Ls$, the $\level-$sublevel set $\Ls$ is regular\footnote{
		The set $C$ is regular at $x$ provided the Bouligand tangent cone at $x$ of $C$ coincides with the Clarke tangent cone at $x$ of $C$ (see \cite[Definition 2.4.6]{clarke1990optimization}). Furthermore, every convex set is regular -- see \cite[Theorem 2.4.7 and (page 55) and Corollary 2 (page 56)]{clarke1990optimization} for other special cases of regular sets.}
	at $x$ by a direct application of \cite[Corollary 2 of Theorem 2.4.7 (page 56)]{clarke1990optimization} with $f(x) = \Ly(x) - \level$.
	Moreover, since \eqref{eq:lya1} and item \ref{item:Ly1} hold, for each $x \in \Ly^{-1}(\level)$, $F(x) \subset T_{\Ls}(x)$, and the set $\Ls$ admits a hypertangent\footnote{See \cite[Section 2.4]{clarke1990optimization}.} at every $x$ \chai{applying Lemma~\ref{lem:innerTcone}.}\footnote{Function $h(x) = \Ly(x) - \level$ is directional Lipschitz since $\Ly$ is continuously differentiable and by item (i) in \cite[Theorem 2.9.4]{clarke1990optimization}.}
	Then, we show that \eqref{eq:intersect} holds for every $x \in \partial (\K \cap C) \setminus D$ in the following cases:
	\begin{enumerate}[leftmargin= 4 mm]
		\item \textit{For every $x \in (\inter \Ls \cap \partial C)\setminus D$,}
		since $T_{\Ls \cap C}(x) = T_C(x)$, \ref{item:Ly2} implies \eqref{eq:intersect} holds;
		\item \textit{For every $x \in \Ly^{-1}(\level) \cap \inter C$,}
		we have $T_{\Ls \cap C}(x) = T_{\Ls}(x)$. This implies \eqref{eq:intersect} holds for every such $x$, because $F(x) \subset T_{\Ls}(x)$ as shown above;
		\item \textit{For every $x \in (\Ly^{-1}(\level) \cap \partial C)\setminus D$,}
		\ref{item:Ly3} implies
		\begin{align}\label{eq:int}
			T_C(x) \cap \inter T_{\Ls}(x) \neq \emptyset.
		\end{align}
		Then, since $\Ls$ and $C$ are regular at $x$, we can apply \cite[Corollary 2 of Theorem 2.9.8 (page 105)]{clarke1990optimization} with
		$C_1 = C$ and $C_2 = \Ls$ since $\Ls$ admits a hypertangent at $x$: for every $x \in (\Ly^{-1}(\level) \cap \partial C) \setminus D$, we have
		\vspace{-3mm}
		$$T_C(x) \cap T_{\Ls}(x) = T_{\Ls \cap C}(x),$$
		i.e., \eqref{eq:intersect} holds.
	\end{enumerate}
	Hence, condition \ref{item:wfi2} in Theorem~\ref{thm:wfi} holds for the sets $C, K = \K$ and the map $F$.
	
	Moreover, \eqref{eq:lya2} implies for every $x \in \K \cap D$, $G(x)\subset \Ls$.
	Together with item \ref{item:Ly5}, \eqref{eq:lya2} leads to condition \ref{item:wfi1} in Theorem~\ref{thm:wfi}.
	Then, according to Theorem~\ref{thm:wfi}, $\K$ is weakly forward invariant for $\HS$ as condition \ref{item:N*} holds by item \ref{item:Ly4}.
	
	For the remainder of the proof, we show that $\K$ is forward invariant when condition \ref{item:Ly6} holds.
	First, we prove $\K$ is forward pre-invariant for the hybrid system $\HS$.
	
	Consider the restriction to hybrid system $\HS$ to the set $\Lss$, denoted $\wHS$ and whose data is $(\wC,F,\wD,G)$, where the flow set and the jump set are given by $\wC = \Lss \cap C$ and $\wD = \Lss \cap D$, respectively.
	Note that \eqref{eq:lya2} implies for every $x \in \K \cap D$, $G(x) \subset \Ls$.
	Then, every $\phi \in \sol_{\wHS}(\K)$ has $\rge \phi \subset \Ls$ if $\phi$ cannot leave $\Ls$ by ``flowing.''
	We show by contradiction that this is the case.
	Suppose $\phi$ left $\Ls$ by ``flowing'' during the interval $I^{j^*} : = [t_{j^*}, t_{j^* + 1}]$: namely, $\phi$ left $\Ls \cap C$ and entered $(\Lss \setminus \Ls) \cap C$.
	More precisely, since $\Ls \subsetneq \Lss$, by closedness of $\K$ and item \ref{item:S1} in Definition~\ref{def:solution}, there exist hybrid time instants $(t^*,j^*), (\tau^*, j^*) \in \dom \phi$ with $\phi(t^*,j^*) \in (\Lss \setminus \Ls) \cap C$, $\phi(\tau^*,j^*) \in \Ly^{-1}(\level) \cap C$, and $\phi(t, j^*) \in (\Lss \setminus \Ls) \cap C$ for all $t \in (\tau^*, t^*]$, where $t_{j^*} < \tau^*< t^* \leq t_{j^*+ 1}$.
	Hence, we have
	\ConfSp{\ComSp}
	\begin{align}\label{eq:contra-flow}
		\Ly(\phi(\tau^*,j^*)) = \level < \Ly(\phi(t^*,j^*)) \leq \rU.
	\end{align}
	By item \ref{item:S1} in Definition~\ref{def:solution}, for every $t \in \inter I^{j^*}$, $\phi(t,j^*) \in \wC$.
	According to \eqref{eq:lya1}, $\frac{d}{dt} \Ly(\phi(t,j^*)) \leq 0$ for almost all $t \in I^{j^*}$.
	Then, integrating both sides, we have
	$$\Ly(\phi(t^*, j^*)) \leq \Ly(\phi(\tau^*, j^*)),$$
	which contradicts with \eqref{eq:contra-flow}.
	Hence, every $\phi \in \sol_{\wHS}(\K)$ stays in $\K$ during flow.
	Therefore, if $\phi$ left $\K$ and entered $\Ls \setminus \K$, which is outside of $C \cup D$ by definition of $\K$, it must have left $C \cup D$ via jumps.
	This is not possible by virtue of \ref{item:Ly6}.
	Thus, we establish the forward pre-invariance of $\K$ for $\wHS$ by Definition~\ref{def:FI}.

	Moreover, we verify that every $\phi \in \sol_{\wHS}(\K)$ with $\rge \phi \subset \K$ is also a maximal solution to $\HS$ by contradiction.
	Suppose there exists $\phi \in \sol_{\wHS}(\K)$ with $\rge \phi \subset \K$ that can be extended outside of $\K$ for $\HS$.
	More precisely, there exists $\psi \in \sol_{\HS}(\K)$, such that $\dom \psi \setminus \dom \phi \neq \emptyset$, for every $(t,j) \in \dom \phi, \psi(t,j) = \phi(t,j)$ and for every $(t,j) \in \dom \psi \setminus \dom \phi$, $\psi(t,j) \notin \K$.
	Let $(T,J) = \sup \dom \phi$. We have two cases:
	\begin{enumerate}[resume, leftmargin= 4mm]
		\item $\psi$ extends $\phi$ via flowing: namely, $\psi(T,J) = \phi(T,J) \in \K \cap C$, $t \mapsto \psi(t,J)$ is absolute continuous on $I^J$.
		By item \ref{item:S1} in Definition~\ref{def:solution}, $\psi(t,J) \in C$ for all $t \in \inter I^J$. 
		Thus, it must be the case that $\psi(t,J) \in C\setminus \Ls$ for some $t \in I^J$.
		Since $\Ls \subsetneq \Lss$, there exists $t^* \in I^J$ such that $\psi(t^*,J) \in \Lss \cap (C\setminus \Ls)$.
		This contradicts with the maximality of $\phi$ to $\wHS$.
		\item $\psi$ extends $\phi$ via jumping: namely, $\psi(T,J) = \phi(T,J) \in \K \cap D$ and $\psi(T, J+1) \notin \K$.
		By item \ref{item:S2} in Definition~\ref{def:solution}, this contradicts with the maximality of $\phi$ to $\wHS$.
	\end{enumerate}
	
	To complete the proof for forward invariance of $\K$ for $\HS$, we show that every $\phi \in \sol_{\HS}(\K)$ is also complete.
	Because condition \ref{item:Ly6} implies \ref{item:Ly5}, we know the set $\K$ defined by the chosen $\level < \rU$ is weakly forward invariant for $\HS$.
	Hence, there exists a nontrivial solution to $\HS$ from every $x \in \K$.
	Case \ref{item:b.1} Proposition~\ref{prop:existence} is excluded for every $\phi \in \sol_{\HS}(\K)$ since $\K \cap C$ is a closed set.
	Case \ref{item:b.2} is not possible for every maximal solutions from $\K$ by assumption \ref{item:Ly4}.
	Finally, $G(\K \cap D) \subset \K$ implies case \ref{item:c} in Proposition~\ref{prop:existence} does not hold. Therefore, only case \ref{item:a} is true for every maximal solution starting from $\K$.
\end{proof}

Condition \ref{thm:Ly}.3) together with \eqref{eq:lya1} result in a less restrictive requirement on the flow map $F$ when compared to the usual Lyapunov conditions for stability purposes, for instance, condition (3.2b) in \cite[Theorem 3.18]{65}, which often rely on finding a qualified positive definite function with strict decrease outside the set to stabilize.
It is not a trivial task to relax condition \ref{thm:Ly}.3) in Theorem~\ref{thm:Ly}.
When the set $\{\xi \in F(x) : \inner{\nabla \Ly(x)}{\xi} < 0\}$ is empty for some $x \in \Ly^{-1}(\level) \cap C$, we have that for every $\xi \in F(x)$, $ \inner{\nabla \Ly(x)}{\xi} = 0.$
With item \ref{item:Ly1}, it is either that $F(x) = 0$ or $F(x) \neq 0$.
If the former holds, condition \ref{item:wfi2} in Theorem~\ref{thm:wfi} holds trivially.
However, if the latter holds, it is possible to get $F(x) \cap T_{\Ls \cap C} = \emptyset$ at such $x$, which implies that only a trivial solution exists at such $x$.
The following example illustrates such a case.
\begin{example}
	Consider a system on $\reals^2$ given by
	$\dot x = F(x) : = (x_2, - x_1)$ with $C = (-\infty, -1] \times \reals$
	and pick $\Ly$ as $\Ly(x) = x^2$ with $\rU = 2$. $\K$ is nonempty and closed for $r \in [1, \rU)$.
	The conditions in Theorem~\ref{thm:Ly} except for \ref{thm:Ly}.3), which does not hold for $r = 1$. In fact, for $r = 1$, at the point $(-1,0)$, the vector $F((-1,0)) = (0,1)$ lays in $T_C((-1,0))$ and satisfies $\inner{\nabla \Ly(x)}{F(x)} = 0$ for each $x \in \Ls \cap C$, so \ref{thm:Ly}.3) does not hold. As a result $F((-1,0)) \notin T_{\Ls \cap C}((-1, 0))$.
\end{example}


\begin{remark}
	As stated in Section~\ref{sec:intro}, invariance is also a property that is key in the study of safety in dynamical systems.
	The Lyapunov-like function approach in this section resembles the idea behind the safety certificates.
	Note that the function $\Ly$ in the results in this section is not sign definite and that the aim was to assume as few properties as possible, though it should be recognized that the invariance property obtained is only for its sublevel sets.
	Connections between results in this section and their extensions to invariance-based control design is the focus of the upcoming second part of this paper.
\end{remark}

\section{Forward Invariance Analysis for a Controlled Single-Phase DC/AC Inverter System}
\label{sec:inverter}
We devote this section to present an application for the proposed forward invariance analysis tools in this work.
The system of interest is a controlled single-phase DC/AC inverter; see \cite{91} for complete design details.
As shown in Fig.~\ref{fig:inverter}, the inverter consists of a full H-bridge connected to a series RLC filter.
The dynamics of the system are
\begin{align}\label{eq:plant}
	\centering
	&\begin{bmatrix}
		\dot{i}_L \\ \dot{v}_C
	\end{bmatrix}
	=  f_q(\z): =
	\begin{bmatrix}	
		\frac{\vdc}{L} q -\frac{R}{L} i_L -\frac{1}{L} v_C\\
		\frac{1}{\Ca} i_L
	\end{bmatrix},
\end{align}
where $R,L,\Ca$ are parameters of the circuit, $\z := (i_L, v_C) \in \reals^2$,  and $q \in Q:= \{-1,0,1\}$ is a logic variable that describes the position of the switches.
\begin{figure}[H]
	\centering
	\includegraphics[width=\figsize]{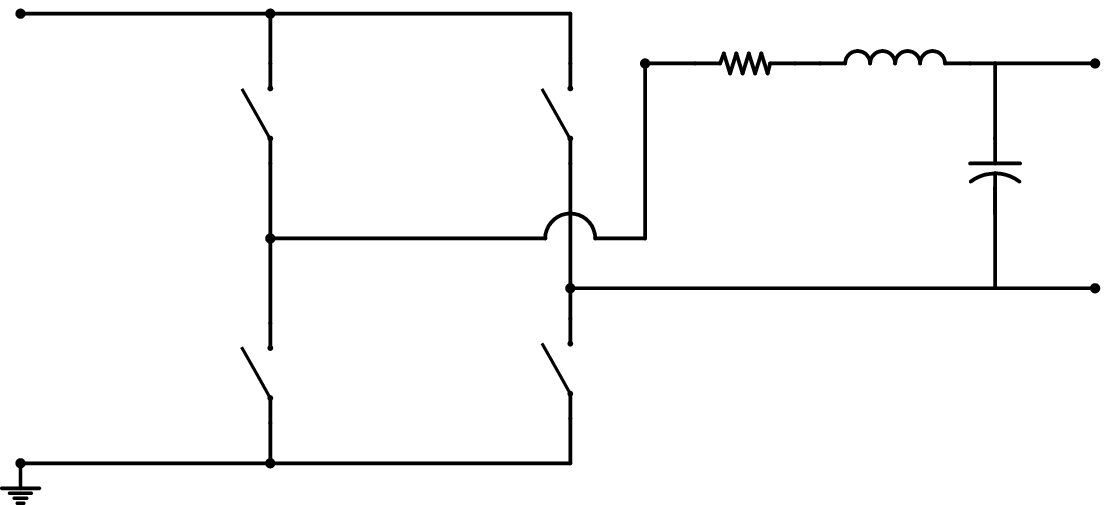}
	\small
	\put(-1.02,.24){$V_{in}$}
	\put(-1,.38){$+$}
	\put(-1,.08){$-$}
	\put(-.03,.29){$v_C$}
	\put(-.035,.35){$+$}
	\put(-.035,.22){$-$}
	\put(-.21,.33){$i_L$}
	\put(-.25,.38){\vector(1,0){0.1}}
	\put(-.35,.43){$R$}
	\put(-.21,.43){$L$}
	\put(-.2,.27){$\Ca$}
	\put(-.45,.42){$V_{in}$}
	\put(-.83,.33){$S_1$}
	\put(-.83,.1){$S_2$}
	\put(-.56,.1){$S_3$}
	\put(-.56,.33){$S_4$}
	\caption{Single-phase DC/AC inverter.}
	\label{fig:inverter}
\end{figure}
By controlling their position (``ON'' or ``OFF''), the voltage $\vin$ to the RLC filter equals to either $\vdc$ , 0 or $-\vdc$.
More precisely, we denote by $q = 1$ when $S_1=S_3= $ ON and $S_2=S_4=$ OFF; by $q = 0$ when $S_1=S_4= $ OFF and $S_2=S_3=$ ON; by $q = -1$ when $S_1=S_3= $ OFF and $S_2=S_4=$ ON.
Given system parameters $\vdc$, $L$, $\Ca$ and $R$, hybrid controller is designed to generate a sinusoidal-like output $v_C$ approximating a reference voltage $\vr$ with amplitude $b$ and frequency $\omega$ of the form $t \mapsto \vr(t) = b\sin{(\omega t + \theta)},$ where $\theta$ is the initial phase.
It can be shown that such a reference signal on the $(i_L, v_C)$ plane has to make
\begin{align}\label{eq:inverterV}
	\Ly(\z) := \left(\frac{i_L}{a}\right)^2+\left(\frac{v_C}{b}\right)^2,
\end{align}
unitary when $v_C= \vr$ and $i_L = \Ca\dot{v}_C$, where $a = \Ca\omega b$.
Let $x = (q, \z)$ and given parameters $0 < c_i < 1< c_o$, this control goal requires rendering the band around the reference trajectory given by
$${\cal T} = \{x \in Q \times \reals^2: c_i \leq \Ly(\z) \leq c_o\},$$
forward invariant for the closed-loop system.
Then, the precision of the approximation is tunable based on the two design parameters $c_i $ and $c_o$.
As $c_o - c_i \rightarrow 0$, the resulting closed-loop trajectories are ``closer'' to the reference trajectory, which make $\Ly$ in \eqref{eq:inverterV} unitary.
Using the hybrid controller proposed in \cite{91}, the closed-loop system, denoted $\HS = \data$, is in form of \eqref{eq:H} with system data given by $F(x) = \left(0, f_q(\z)\right)$ for every $x \in C : = \cal T$, and $G(x) = (G_q(\z), \z)$ for every $x \in D$, with\footnote{With $\epsilon$ as a (small enough) positive parameter,
	$M_1 = \{\z\in \reals^2: V(\z) = c_o, 0 \leq i_L \leq \epsilon, v_C \leq 0\}$ and
	$M_2 = \{\z\in \reals^2: V(\z) = c_o, -\epsilon \leq i_L \leq 0, v_C \geq 0\}.$}
	\IfConf{
		\begin{align*}
		\centering
		G_q(\z) :=
		\begin{cases}
		-1 & \ifeq q \neq -1 \mbox{ and }\\
		& ((V(\z) = c_o \text{ and } i_L \geq 0 \text{ and } \z \notin M_1)\\
		& \text{ or } (V(\z) = c_i \text{ and } i_L \leq 0));\\
		0  & \ifeq (\z \in M_1 \text{ and } i_L \neq \epsilon \text{ and } q  = 1)\\
		& \text{or }  (\z \in M_2 \text{ and } i_L \neq -\epsilon \mbox{ and }q  = -1);\\
		1  & \ifeq q \neq 1 \mbox{ and }\\
		& ((V(\z) = c_o \text{ and } i_L \leq 0 \mbox{ and } \z \notin M_2)\\
		& \text{ or } (V(\z) = c_i \text{ and } i_L \geq 0));\\
		\{0, 1\} & \ifeq (V(\z) = c_o, i_L = -\epsilon, v_C \geq 0);\\
		\{-1, 0\} & \ifeq ( V(\z) = c_o, i_L = \epsilon, v_C \leq 0);
		\end{cases}
		\end{align*}
	}{
	\begin{align*}
		G_q(\z) :=
		\begin{cases}
			-1 & \ifeq q \neq -1 \mbox{ and }((V(\z) = c_o \text{ and } i_L \geq 0 \text{ and } \z \notin M_1)\text{ or } (V(\z) = c_i \text{ and } i_L \leq 0));\\
			0  & \ifeq (\z \in M_1 \text{ and } i_L \neq \epsilon \text{ and } q  = 1) \text{ or }  (\z \in M_2 \text{ and } i_L \neq -\epsilon \mbox{ and }q  = -1);\\
			1  & \ifeq q \neq 1 \mbox{ and }((V(\z) = c_o \text{ and } i_L \leq 0 \mbox{ and } \z \notin M_2) \text{ or } (V(\z) = c_i \text{ and } i_L \geq 0));\\
			\{0, 1\} & \ifeq (V(\z) = c_o, i_L = -\epsilon, v_C \geq 0);\\
			\{-1, 0\} & \ifeq ( V(\z) = c_o, i_L = \epsilon, v_C \leq 0),
		\end{cases}
	\end{align*}
}
and the jump set $D$ given as
\begin{align*}
	\begin{array}{rlc}
		D:=
		& \{x \in Q \times \reals^2 : V(\z) = c_i, i_L q \leq 0, q\neq 0\}\\
		& \bigcup \{x \in Q \times \reals^2 : V(\z) = c_o, i_L q \geq 0, q \neq 0 \}\\
		& \bigcup\{x \in Q \times \reals^2 : V(\z) = c_i, q = 0\}.
	\end{array}
\end{align*}

With $\alpha = \frac{2}{a^2L}$ and $\beta = \frac{2}{b^2\Ca}$, we define $\Gamma = \{ z \in \reals^2: -\alpha \vdc \leq - \alpha Ri_L + (\beta - \alpha)v_C \leq \alpha \vdc \}$.
Next, applying Theorem~\ref{thm:fi}, we show the set $\cal T$ is forward invariant for $\HS$.
\begin{proposition}(\cite[Proposition 1]{91})\label{prop:1}
	Given positive system constants $R, L, \Ca, \omega, \vdc $ such that $L\Ca\omega^2 >1$, and $0 < c_i < 1 < c_o$ such that  ${\cal T} \subset Q \times \Gamma$, the closed set $\cal T$ is forward invariant for the closed-loop system $\HS = \data$.
\end{proposition}
\begin{proof}
	By observation, $\cal T$, $C, D, F$ satisfy Assumption~\ref{assum:data}, and $F$ is continuously differentiable.
	Condition \ref{item:fi1} in Theorem~\ref{thm:fi} holds since $G(x) : = (G_q(\z) \times \{\z\})  \subset (Q \times \{\z\})$ for every $x \in \cal T$.
\chai{
	Since ${\cal T} = C$, $T_{{\cal T} \cap C}(x) = T_C(x)$ for every $x \in C$.
	Hence, item \ref{item:fi2} in Theorem~\ref{thm:fi} holds if $F(x) \in T_C(x)$ for every $x \in \partial C \setminus L$.
	Note that $L \supset D$ by applying  Lemma~\ref{lem:innerTcone} and the inner product properties listed in \cite[Lemma 2]{91}.\footnote{In fact, $L \cap {\cal T} \subset D$ by construction.}
	In particular, for every $x \in D \cap \{ x : V(\z) = c_o \}$, apply item~\ref{item:inner2} in Lemma~\ref{lem:innerTcone} with $S = L_V(c_o)$, $h(x) = V(z) - c_o$; while for every $x \in D \cap \{ x : V(\z) = c_i \}$, consider $S = L_{(-V)}(-c_i)$, $h(x) = -V(z) + c_i$ and item~\ref{item:inner1} in Lemma~\ref{lem:innerTcone}.}
	Then, for every $x \in \partial {\cal T}\setminus L$, by applying Lemma~\ref{lem:innerTcone} and the inner product properties listed in \cite[Lemma 2]{91}, we have the value of $q$ chosen by the designed controller always results in one of the two cases below:
		\begin{itemize}
			\item When $x \in \{ x : V(\z) = c_o \}\setminus L$, $F(x) \in \{0\} \times T_{L_V(c_o)}(\z)$;
			\item When $x \in \{ x : V(\z) = c_i\}\setminus L$, $F(x) \in \{0\} \times \overline{\reals^2 \setminus T_{L_V(c_i)}(\z)}$.
		\end{itemize}
	Hence, $F(x) \in T_C(x)$ for every $x \in {\partial \cal T}\setminus L$.

	Finally, applying Lemma~\ref{lem:star}, item \ref{item:N*} holds since $F$ is affine linear for each assigned $q \in Q$.
\end{proof}

Numerical simulations are performed to verify the property listed in Proposition~\ref{prop:1} via MATLAB Hybrid Equations Toolbox (HyEQ); see details in \cite{74}.
The following system parameters are used:
$R = 1\Omega$, $L = 0.1H$, $\Ca = 66.6\mu F$,
$\vdc \equiv 220V$, $b = 120V$, $\omega = 120 \pi$,
and $c_i = 0.9, c_o = 1.1$.
\begin{figure}[H]
	\centering
		\setlength{\unitlength}{0.45\textwidth}
		\includegraphics[width=\unitlength]{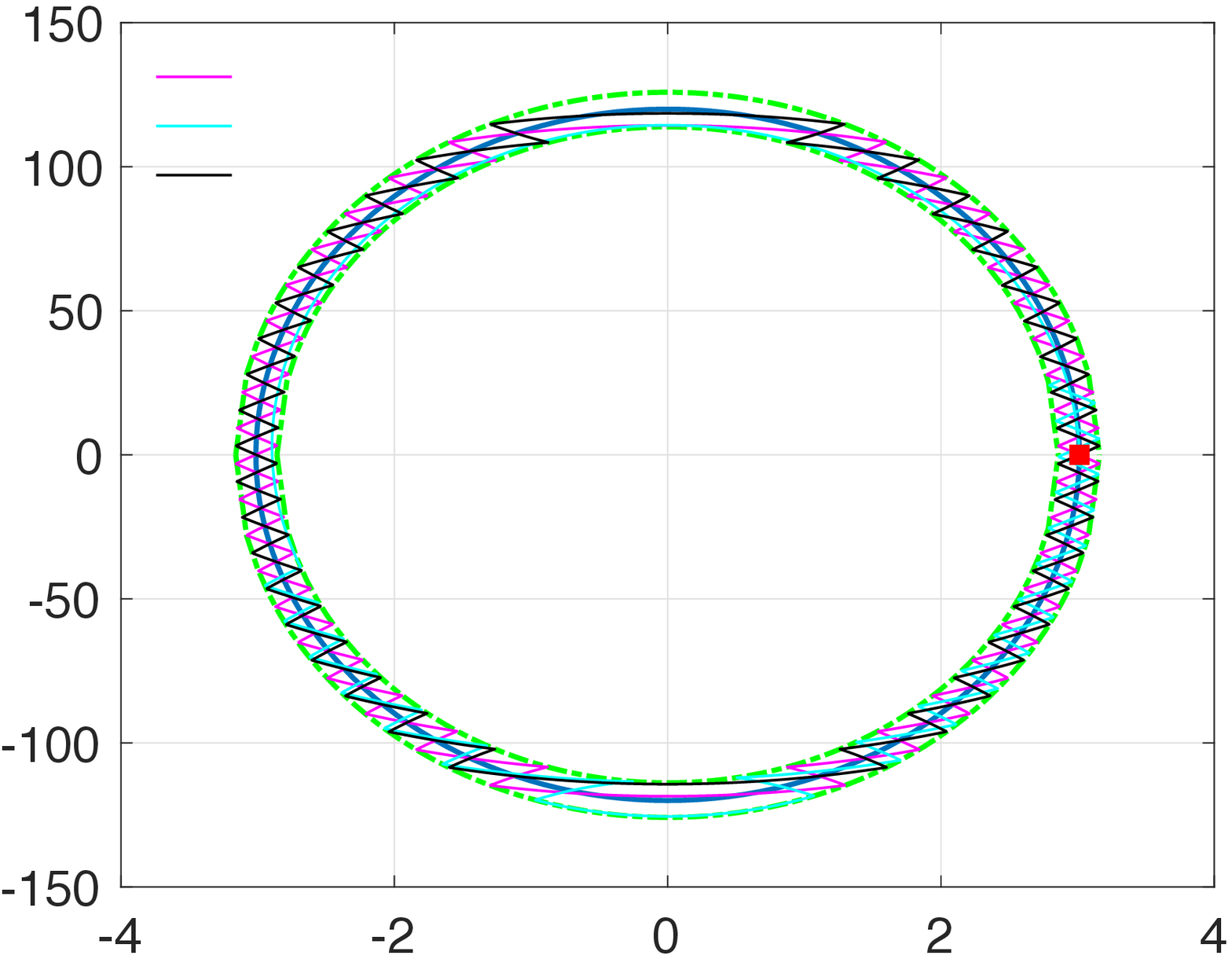}
		\put(-1.03,0.4){$v_C$}
		\put(-0.5,-0.03){$i_L$}
		\put(-0.35,0.36){$z_0$}
		\put(-0.31,0.37){\vector(1,0){0.1}}
		\put(-0.82,0.65){\tiny $q_0 = -1$}
		\put(-0.82,0.61){\tiny $q_0 = 0$}
		\put(-0.82,0.57){\tiny $q_0 = 1$}
		\caption{Simulations of $\HS$ different initial values of $q$.}
		\label{fig:simu511}
\end{figure}
Figure~\ref{fig:simu511} shows the solutions to the closed-loop system $\HS$ with $\z_0 = (b\Ca\omega, 0) =  (3.013, 0)$ and $q_0$ as either $-1,0$ or $1$.
The three solutions stay within the projection of $\cal T$ onto the $(i_L, v_C)$ plane.
\begin{figure}[H]
	\centering
		\setlength{\unitlength}{0.45\textwidth}
		\includegraphics[width=\unitlength]{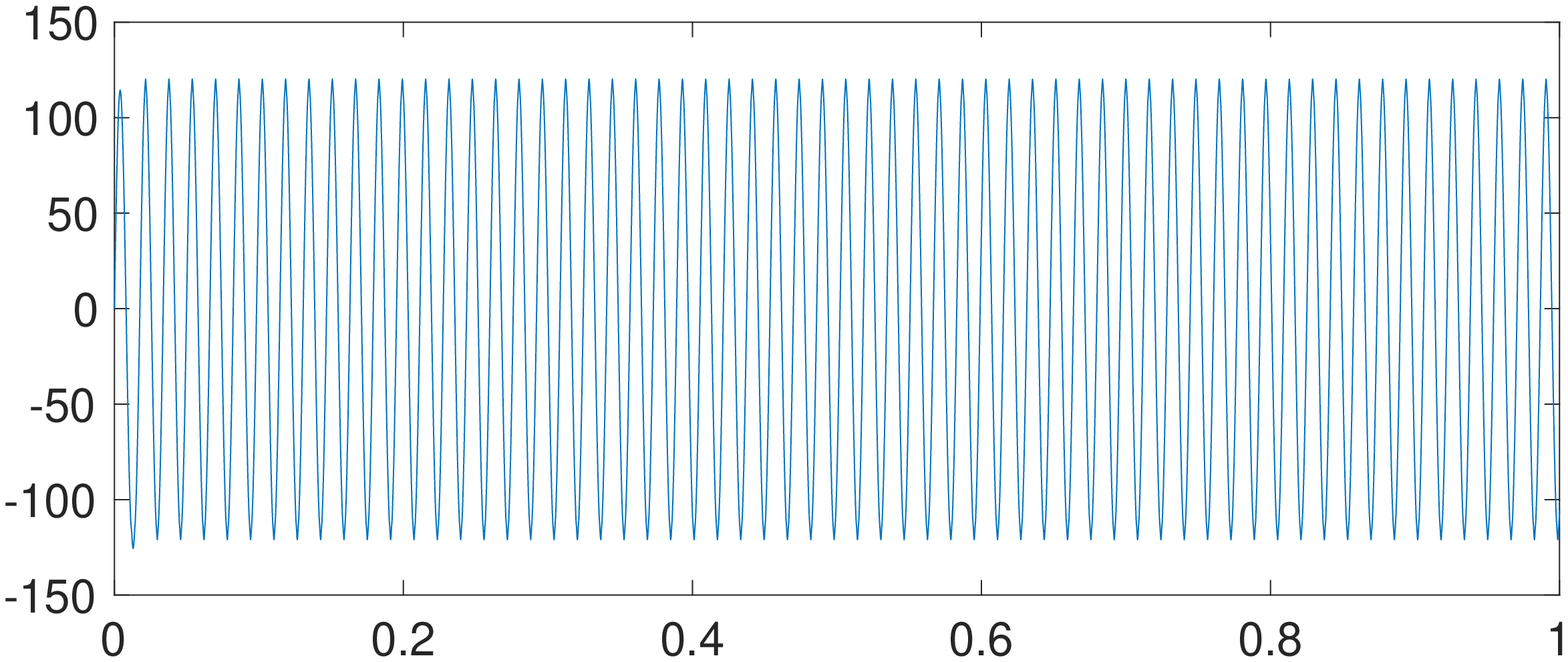}
		\put(-1.05,0.23){$v_C$}
		\put(-0.55,-0.05){time[$s$]}
		\caption{The output $v_C$ signal of $\HS$ with $q_0 = 0$.}
		\label{fig:simu_vdc}
\end{figure}
Fig.~\ref{fig:simu_vdc} shows that the output voltage $v_C$ behaves sinusoidal-like.
A Fast Fourier Transform (FFT) analysis is also performed to show that the output signal has the desired frequency; see \cite[Fig. 6a]{91}.

\section{Conclusion}
Forward invariance properties of sets that are uniform over the disturbances for hybrid inclusions with disturbance inputs are studied in this paper.
When a set $K$ enjoys such a property, solution pairs evolve within the set they started from, even under the effect of disturbances.
We formally define robust forward invariance of sets for hybrid systems $\Hw$ modeled using differential and difference inclusions with state and disturbance constraints.
Among the four notions, two of them are considered stronger than the other two in the sense that all maximal solution pairs that start from the set of interest stay in it.

Sufficient conditions for each notion are presented in terms of the data of the hybrid system and require checking conditions involving its discrete and continuous dynamics within $K$.
In particular, when starting from an intersection involving $K$ and the jump set, the jump map ought to map the state back to $K$ to allow solutions evolving within $K$ during jumps; while when starting from an intersection involving $K$ and the flow set, the flow map needs to have vectors pointing inward of $K$ to allow solutions evolving within it during flows.
To guarantee the robust invariance notions, the flow map is required to enjoy a locally Lipschitz property on $\Cw$ to avoid solutions from leaving $K$.
Such a property is also ensured by Assumption~\ref{assum:wbound}--a mild assumption on the inputs allowed by $\Cw$ at points that are at the boundary of $K$ and in $\Cw$.

To achieve the notions that require completeness of maximal solution pairs, a general result for $\Hw$ that characterizes all possible ending behaviors of maximal solution pairs is presented.
The existence of nontrivial solution pairs from every point in $K$ is ensured by the assumption of zero disturbances being admissible during flows.
In addition,  condition \textit{$\star$)} is introduced to exclude the case of solution pairs escaping to infinity in finite time, where two solution-independent conditions are proposed to verify \textit{$\star$)} for $\Hw$.
The results on robust invariance are specialized to the nominal case, i.e., $\HS$.

Finally, one result to render sublevel sets of Lyapunov-like functions nominal forward invariant for $\HS$ is provided.
In particular, properties of the tangent cone to the sublevel of the Lyapunov-like function is exploited to derive mild conditions for invariance of the sublevel set.
These properties hold for free when the flow set is convex (locally) or the Lyapunov-like functions strictly decrease at the boundary of the set $K$.

In a second part of this work, which is being prepared, results on existence of invariance inducing state-feedback laws for $\Hw$ using robust control Lyapunov functions for forward invariance will be provided.
In particular, controller synthesis that constructs state-feedback laws using a pointwise minimum norm selection scheme are under development.
Future research directions include studying optimality properties of feedback laws via inverse optimality analysis and the development of barrier certificates for hybrid systems $\HS$ and $\Hw$.
	
	\bibliographystyle{unsrt}
	\bibliography{bibs/CJref,bibs/RGSweb,bibs/CJintro}
	
	\appendix \label{sec:app}
	\subsection{Auxiliary Definitions and Results for Hybrid Systems}
\begin{definition}(solutions to $\HS$, \cite[Definition 2.6]{65})\label{def:solution}
	A hybrid arc $\phi$ is a solution to the hybrid system $(C, F, D, G)$ if $\phi(0,0) \in \overline{C} \cup D$, and
	\begin{enumerate}[label = (S\arabic*), leftmargin=2.3\parindent]
		\item\label{item:S1} for all $j \in \mathbb{N}$ such that $I^j$ has nonempty interior
		\begin{align*}
		&\phi(t,j) \in C &\mbox{for all } &\quad t \in \inter I^j,\\
		&\frac{d\phi}{dt}(t,j) \in F(\phi(t,j)) &\mbox{for almost all } &\quad t \in I^j,
		\end{align*}
		\item\label{item:S2} for all $(t,j) \in \dom\phi$ such that $(t, j+1) \in \dom\phi$,
		$$ \phi(t,j) \in D \qquad \phi(t,j+1) \in G(\phi(t,j)).$$
	\end{enumerate}
\end{definition}
\begin{proposition}(\cite[Proposition 2.2]{120})\label{prop:existence}
	Consider the hybrid system $\HS = (C,F,D,G)$. Let $\xi \in \overline{C} \cup D$. If $ \xi \in D$ or
	\begin{enumerate}[label = (VC), leftmargin= 8mm]
		\item\label{item:VC} there exist $ \varepsilon > 0$ and an absolutely continuous function $z: [0,\varepsilon] \rightarrow \reals^n$ such that $ z(0) = \xi, \dot{z}(t) \in F(z(t))$ for almost all $t \in [0,\varepsilon]$ and $z(t)\in C$ for all $t \in (0,\varepsilon]$,
	\end{enumerate}
	then there exists a nontrivial solution $\phi$ to $\HS$ with $\phi(0,0) = \xi$. If \ref{item:VC} holds for every $ \xi \in \overline{C} \setminus D$, then there exists a nontrivial solution to $\HS$ from every point of $ \overline{C} \cup D,$ and every $\phi \in \mathcal{S}_{\HS}$ satisfies exactly one of the following:
	\begin{enumerate}[label = \alph*), leftmargin= 4mm]
		\item $\phi$ is complete;
		\item $\phi$ is not complete and ``ends with flow'', with $(T,J) = \sup\dom\phi$, the interval $I^J$ has nonempty interior; and either
		\begin{enumerate}[label = b.\arabic*)]
			\item $I^J$ is closed, in which case $\phi(T,J) \in \overline{C}\setminus(C \cup D)$; or
			\item\label{item:b.2n} $I^J$ is open to the right, in which case $(T,J) \notin \dom\phi$, and there does not exist an absolutely continuous function $z : \overline{I^J} \rightarrow \reals^n$ satisfying $ \dot{z}(t) \in  F(z(t))$ for almost all $t \in I^J, z(t) \in C$ for all $t \in \mbox{\normalfont int }I^J$, and such that $z(t) = \phi(t,J)$ for all $t \in I^J $;
		\end{enumerate}
		\item\label{item:cn} $\phi$ is not complete and ``ends with jump'': for  $(T,J) = \sup\dom\phi$, one has $\phi(T,J) \notin \overline{C} \cup D$.
	\end{enumerate}
	Furthermore, if $G(D) \subset \overline{C} \cup D,$ then \ref{item:cn} above does not occur.
\end{proposition}

\subsection{Auxiliary Definitions and Results for Set-valued Maps}


\begin{definition}(outer semicontinuity of set-valued maps)\label{def:osc}
	A set-valued map $S: \reals^n \rightrightarrows \reals^m$ is outer semicontinuous at $x \in \reals^n$ if for each sequence $\{x_i\}^{\infty}_{i=1}$ converging to a point $x \in \reals^n$ and each sequence $y_i \in S(x_i)$ converging to a point $y$, it holds that $y \in S(x)$; see  \cite[Definition 5.4]{rockafellar2009variational}. Given a set $K \subset \reals^n$, it is outer semicontinuous relative to $K$ if the set-valued mapping from $\reals^n$ to $\reals^m$ defined by $S(x)$ for $x\in K$ and $\emptyset$ for $x \notin K$ is outer semicontinuous at each $x\in K$.
\end{definition}

	\begin{definition}(Lipschitz continuity of set-valued maps)\label{def:lip}
		Given a set-valued map $F: \reals^n \times \Wc \rightrightarrows \reals^n$, the mapping $x \mapsto F(x, w)$ is locally Lipschitz uniformly in $w$ at $x$, if there exists a neighborhood $U$ of $x$ and a constant $\lambda \geq 0$ such that for every $\xi \in U$
		\begin{align*}
			F(x, w) \subset F(\xi, w) + &\lambda |x - \xi|\ball\\
			&\forall w \in \{w \in \Wc : (U \times \Wc) \cap \dom F\}.
		\end{align*}
		Furthermore, $x \mapsto F(x, w)$ is locally Lipschitz uniformly in $w$ on set $K \subset \dom F$ when it is locally Lipschitz uniformly in $w$ at each $x \in \Pwc{K}$.
	\end{definition}

\chai{
\begin{lemma}(\cite[Theorem 2.9.10]{clarke1990optimization})\label{lem:innerTcone}
	Given a set $S : = \{x: h(x) \leq 0\}$,
	suppose that, for every $x \in \{x: h(x) = 0\}$, $h$ is directionally Lipschitz at $x$ with $0 \notin \nabla h(x) \neq \emptyset$ and the collection of vectors $Y : = \{y : \inner{\nabla h(x)}{y} <\infty \}$ is nonempty.
	Then, the set $S$ admits a hypertangent at $x$ and
	\begin{enumerate}[leftmargin = 4mm, label = \arabic*)]
		\item\label{item:inner1} $y \in T_S(x) \ \ \ifeq \inner{\nabla h(x)}{y} \leq 0$;
		\item\label{item:inner2} $\exists y \in \inter T_S(x)\cap \inter Y$
		such that  $\inner{\nabla h(x)}{y} < 0$.
	\end{enumerate}
\end{lemma}
}

\end{document}